\documentclass[a4paper, 11pt]{amsart}   
\usepackage{mathptmx, amssymb,amscd,latexsym, eulervm}   
\usepackage{amsmath}
\usepackage{amsthm}
\usepackage{mathdots}
\usepackage[pagebackref,colorlinks=true,linkcolor=blue,urlcolor=blue]{hyperref}
\usepackage{color}
\usepackage[onehalfspacing]{setspace}
\usepackage{tabularx}
\usepackage{amsfonts}
\usepackage{paralist}
\usepackage{aliascnt}
\usepackage[initials, lite]{amsrefs}
\usepackage{amscd}
\usepackage{blkarray}
\usepackage{mathbbol}
\usepackage{setspace}
\usepackage[inner=2.4cm,outer=2.4cm, bottom=3.2cm]{geometry}
\usepackage{tikz, tikz-cd}
\usepackage{calligra,mathrsfs}

\usepackage{tikz}
\usetikzlibrary{matrix}
\usetikzlibrary{arrows,calc}
\allowdisplaybreaks

\BibSpec{collection.article}{%
	+{}  {\PrintAuthors}                {author}
	+{,} { \textit}                     {title}
	+{.} { }                            {part}
	+{:} { \textit}                     {subtitle}
	+{,} { \PrintContributions}         {contribution}
	+{,} { \PrintConference}            {conference}
	+{}  {\PrintBook}                   {book}
	+{,} { }                            {booktitle}
	+{,} { }                            {series}
	+{, vol.} { }                            {volume}
	+{,} { }                            {publisher}
	+{,} { \PrintDateB}                 {date}
	+{,} { pp.~}                        {pages}
	+{,} { }                            {status}
	+{,} { \PrintDOI}                   {doi}
	+{,} { available at \eprint}        {eprint}
	+{}  { \parenthesize}               {language}
	+{}  { \PrintTranslation}           {translation}
	+{;} { \PrintReprint}               {reprint}
	+{.} { }                            {note}
	+{.} {}                             {transition}
	+{}  {\SentenceSpace \PrintReviews} {review}
}
\AtBeginDocument{%
	\def\MR#1{}
}

\makeatletter
\@namedef{subjclassname@2020}{%
	\textup{2020} Mathematics Subject Classification}
\makeatother

\newcommand{\kk}{\mathbb{k}}

\newcommand{\NN}{\normalfont\mathbb{N}}
\newcommand{\ZZ}{\normalfont\mathbb{Z}}
\newcommand{\gin}{{\normalfont\text{gin}}}

\newcommand{\PP}{{\normalfont\mathbb{P}}}
\newcommand{\xx}{{\normalfont\mathbf{x}}}

\newcommand{\yy}{\normalfont\mathbf{y}}
\newcommand{\dd}{{\normalfont\mathbf{d}}}

\newcommand{\mm}{{\normalfont\mathfrak{m}}}
\newcommand{\init}{{\normalfont\text{in}}}

\newcommand{\QQ}{\mathbb{Q}}
\newcommand{\pp}{{\normalfont\mathfrak{p}}}
\newcommand{\aaa}{\mathfrak{a}}

\newcommand{\qqq}{\normalfont\mathfrak{q}}
\newcommand{\nn}{{\normalfont\mathfrak{N}}}
\newcommand{\bn}{{\normalfont\mathbf{n}}}
\newcommand{\supp}{{\normalfont\text{supp}}}

\newcommand{\bm}{{\normalfont\mathbf{m}}}
\newcommand{\ttt}{{\normalfont\mathbf{t}}}
\newcommand{\sss}{{\normalfont\mathbf{s}}}
\newcommand{\nnn}{{\normalfont\mathfrak{n}}}

\newcommand{\fJ}{{\mathfrak{J}}}

\newcommand{\rank}{\normalfont\text{rank}}
\newcommand{\reg}{\normalfont\text{reg}}
\newcommand{\HS}{{\normalfont\text{HS}}}

\newcommand{\CS}{{\rm CS}}

\newcommand{\codim}{{\normalfont\text{codim}}}

\newcommand{\Tor}{{\normalfont\text{Tor}}}

\newcommand{\Ext}{\normalfont\text{Ext}}

\newcommand{\Ker}{\normalfont\text{Ker}}

\newcommand{\Quot}{\normalfont\text{Quot}}

\newcommand{\Ann}{\normalfont\text{Ann}}
\newcommand{\Supp}{\normalfont\text{Supp}}
\newcommand{\Ass}{{\normalfont\text{Ass}}}

\newcommand{\Min}{{\normalfont\text{Min}}}

\newcommand{\ee}{{\normalfont\mathbf{e}}}
\newcommand{\Hom}{\normalfont\text{Hom}}

\newcommand{\bbA}{\mathbb{A}}

\newcommand{\OO}{\mathcal{O}}

\newcommand{\LL}{\mathbb{L}}

\newcommand{\FF}{\normalfont\mathcal{F}}
\newcommand{\HL}{\normalfont\text{H}_{\mm}}
\newcommand{\HH}{\normalfont\text{H}}

\newcommand{\Proj}{\normalfont\text{Proj}}
\newcommand{\Hilb}{{\normalfont\text{Hilb}}}
\newcommand{\Spec}{\normalfont\text{Spec}}

\newcommand{\multProj}{\normalfont\text{MultiProj}}
\newcommand{\msupp}{\normalfont\text{MSupp}}

\def\f0{\mathbf{0}}

\def\ff{\mathbf{f}}

\def\1{\mathbf{1}}

\newtheorem{theorem}{Theorem}[section]

\newtheorem{headthm}{Theorem}

\newaliascnt{headcor}{headthm}

\aliascntresetthe{headcor}

\newaliascnt{headconj}{headthm}

\aliascntresetthe{headconj}

\newaliascnt{corollary}{theorem}
\newtheorem{corollary}[corollary]{Corollary}
\aliascntresetthe{corollary}

\newaliascnt{claim}{theorem}

\aliascntresetthe{claim}

\newaliascnt{lemma}{theorem}
\newtheorem{lemma}[lemma]{Lemma}
\aliascntresetthe{lemma}

\newaliascnt{conjecture}{theorem}

\aliascntresetthe{conjecture}

\newaliascnt{proposition}{theorem}
\newtheorem{proposition}[proposition]{Proposition}
\aliascntresetthe{proposition}

\theoremstyle{definition}
\newaliascnt{definition}{theorem}
\newtheorem{definition}[definition]{Definition}
\aliascntresetthe{definition}

\newaliascnt{notation}{theorem}
\newtheorem{notation}[notation]{Notation}
\aliascntresetthe{notation}

\newaliascnt{example}{theorem}
\newtheorem{example}[example]{Example}
\aliascntresetthe{example}

\newaliascnt{examples}{theorem}

\aliascntresetthe{examples}

\newaliascnt{remark}{theorem}
\newtheorem{remark}[remark]{Remark}
\aliascntresetthe{remark}

\newaliascnt{question}{theorem}

\aliascntresetthe{question}

\newaliascnt{questions}{theorem}

\aliascntresetthe{questions}

\newaliascnt{problem}{theorem}

\aliascntresetthe{problem}

\newaliascnt{construction}{theorem}

\aliascntresetthe{construction}

\newaliascnt{setup}{theorem}
\newtheorem{setup}[setup]{Setup}
\aliascntresetthe{setup}

\newaliascnt{algorithm}{theorem}

\aliascntresetthe{algorithm}

\newaliascnt{observation}{theorem}

\aliascntresetthe{observation}

\newaliascnt{defprop}{theorem}
\newtheorem{defprop}[defprop]{Definition-Proposition}
\aliascntresetthe{defprop}

\def\equationautorefname~#1\null{(#1)\null}
\def\sectionautorefname~#1\null{Section #1\null}
\def\subsectionautorefname~#1\null{\S #1\null}

\def\surjects{\twoheadrightarrow}

\title{Multidegrees, prime ideals, and non-standard gradings}

\author{Alessio Caminata}
\address[Caminata]
{Dipartimento di Matematica, Universit\`a di Genova, via Dodecaneso 35, 16146 Genova, Italy}
\email{alessio.caminata@unige.it}

\author{Yairon Cid-Ruiz}
\address[Cid-Ruiz]
{Department of Mathematics, KU Leuven, Celestijnenlaan 200B, 3001 Leuven, Belgium}
\email{yairon.cidruiz@kuleuven.be}

\author{Aldo Conca}
\address[Conca]
{Dipartimento di Matematica, Universit\`a di Genova, via Dodecaneso 35, 16146 Genova, Italy}
\email{conca@dima.unige.it}

\date{\today}

\begin{document}

	\maketitle

	\begin{abstract}
		We study several properties of multihomogeneous prime ideals.
		We show that the multigraded generic initial ideal of a prime has very special properties, for instance, its radical is Cohen-Macaulay.
		We develop a comprehensive study of multidegrees in  arbitrary positive multigraded settings.
		In these environments, we extend the notion of Cartwright-Sturmfels ideals by means of a standardization technique.
		Furthermore, we recover or extend important results in the literature, for instance: we provide a multidegree version of Hartshorne's result stating the upper semicontinuity of arithmetic degree under flat degenerations, and we give an alternative proof of Brion's result regarding multiplicity-free varieties.
	\end{abstract}

	{
		\hypersetup{linkcolor=black}
		\setcounter{tocdepth}{1}
		\tableofcontents
	}

	\section{Introduction}
	
	This paper is concerned with several aspects of the theory of multidegrees.
	The concept of multidegree provides the right generalization of the degree of a variety to a multiprojective setting, and its study goes back to seminal work by van der Waerden \cite{VAN_DER_WAERDEN} in 1929.
	Multidegrees have found applications in several areas (e.g., algebraic geometry, commutative algebra, combinatorics, convex geometry, and more recently, algebraic statistics), and there is a long list of more recent papers where the notion of multidegree (or mixed multiplicity) plays a fundamental role (see, e.g., \cite{Bhattacharya,VERMA_BIGRAD,HERMANN_MULTIGRAD,trung2001positivity,cidruiz2021mixed, EXPONENTIAL_VARIETIES,KNUTSON_MILLER_SCHUBERT, michalek2020maximum, POSITIVITY, CDNG_MINORS, CDNG_CS_IDEALS,  Huh12, huh2019logarithmic, CS_PAPER, EQ_MDEG_SYM_MAT, TRUNG_VERMA, manivel2020complete, KNUTSON_MILLER_YONG,STURMFELS_UHLER}).

	\smallskip

	The goal of this paper is twofold. 
	Firstly, we concentrate on studying several interesting properties that prime ideals enjoy in a multigraded setting.
	Secondly, we make a comprehensive study of multidegrees in non-standard multigradings. 
	
	For organizational purposes, we divide the Introduction into two shorter subsections.

	\subsection{Prime ideals}
	\hfill
	
	Let $\kk$ be a field and $S = \kk[x_{i,j} \mid 1 \le i \le p, 0 \le j \le m_i]$ be a standard $\NN^p$-graded polynomial ring with grading induced by setting $\deg(x_{i,j}) = \ee_i \in \NN^p$, where $\ee_i$ is the $i$-th standard basis vector.
	Then $S$ is naturally seen as the multihomogeneous coordinate ring of the product of projective spaces $\PP = \PP_\kk^{m_1} \times_\kk \cdots \times_\kk \PP_\kk^{m_p}$. 
	
	Let $P \subset S$ be an $S$-homogeneous prime ideal.
	A fundamental goal of this paper is to study several properties of $P$. 
	Although our results regarding prime ideals do not involve in principle multidegrees, our proofs heavily depend on the use of this concept. 
	The fact that prime ideals are quite special from a multigraded point of view has already been noticed. 
	To highlight some of these we point out: Castillo--Cid-Ruiz--Li--Monta\~no--Zhang's result that the support of the positive multidegrees of $S/P$ forms a discrete polymatroid \cite[Theorem A]{POSITIVITY}, Br\"and\'en--Huh's result that the volume polynomial (a different encoding of multidegrees) of $S/P$ is Lorentzian \cite[Theorem 4.6]{HUH_BRANDEN}, and Brion's result on multiplicity-free varieties \cite{BRION_MULT_FREE}.
	
	Our first main result is regarding \emph{multigraded generic initial ideals}.
	Given a monomial order $>$ on $S$, one may define the multigraded generic initial ideal in analogy with the singly-graded setting (see \autoref{sect_multdeg_primes} for details).

	\begin{headthm}[\autoref{thm_sqrt_gin_CM}]
		\label{thmA}
		{\rm(}$\kk$ infinite{\rm)}
		We have that $\sqrt{\gin_>(P)}$ is a Cohen-Macaulay ideal.
	\end{headthm}

	This result can be seen as yet another manifestation of the fact that the initial ideals of primes are ``special''.
	For example:
	\begin{itemize}[--]
		\item In \cite{KS_INIT_PRIMES}, Kalkbrener and Sturmfels showed that the radical of the initial ideal of a prime is equidimensional and connected in codimension one.
		\item In \cite[Chapter 8]{STURMFELS_GROBNER}, Sturmfels proved that the radical of the initial ideal of a toric ideal is Cohen-Macaulay.
		\item In \cite{HOSTEN_THOMAS}, Ho\c{s}ten and Thomas proved that the associated primes of the initial ideal of a toric ideal satisfy a saturated chain property.
	\end{itemize}
	On the other hand, it is not difficult to find a prime ideal $P$ such that $\init_>(P)$ and its radical are not Cohen-Macaulay (see \autoref{rem_not_CM_init}).
	In this sense, \autoref{thmA} is a sharp result.
	
	\smallskip
	
	We also discovered an interesting (and perhaps unexpected) behavior of the minimal primary components of $\gin_>(P)$.
	Let $\fJ = \{j_1,\ldots,j_k\} \subset [p] = \{1,\ldots,p\}$ be a subset. 
	We denote by $S_{(\fJ)}$ the polynomial subring generated by the variables with degrees $\ee_{j_1},\ldots,\ee_{j_k}$, and we write  $I_{(\fJ)} = I \cap S_{(\fJ)}$ for the corresponding contraction of an $S$-homogeneous ideal $I \subset S$.
	For an ideal $I \subset S$, let $\text{MLength}(I)$ denote the maximal length of the minimal primary components of $I$.
	The following theorem shows that  $\text{MLength}$ cannot increase under the natural projections of $\gin_>(P)$.
	
	\begin{headthm}[\autoref{thm_proj_gin}]
		\label{thmB}
		{\rm(}$\kk$ infinite{\rm)}
		We have the inequality 
		$$
		{\rm MLength}\left(\gin_>\left(P_{(\fJ)}\right)\right) \;\le\; {\rm MLength}\left(\gin_>\left(P\right)\right).
		$$
		Moreover, for every $\pp \in \Min_{S_{(\fJ)}}\left(S_{(\fJ)}/\gin_>(P_{(\fJ)})\right)$, we can find $\qqq \in \Min_S\left(S/\gin_>(P)\right)$ such that the length of the $\pp$-primary component of $\gin_>\left(P_{(\fJ)}\right)$ divides the length of the $\qqq$-primary component of $\gin_>(P)$.
	\end{headthm}
	
	To better appreciate the complexity of the process described in \autoref{thmB}, the reader is referred to \autoref{examp_project_gin}.
	Also, \autoref{thmB} may not hold for non prime ideals (see \autoref{rem_counter_gin_not_prime}).
	
	\smallskip
	
	Let $X = \multProj(S/P) \subset \PP = \PP_\kk^{m_1} \times_\kk \cdots \times_\kk \PP_\kk^{m_p}$ be the integral closed subscheme corresponding to $P$.
	We also consider the behavior of the multidegrees of $X$ under the natural projections. 
	It turns out that our proof of \autoref{thmB} is a consequence of this study of multidegrees under projections.
	The projection corresponding to $\fJ = \{j_1,\ldots,j_k\}$ is given by $\Pi_\fJ : \PP = \PP_\kk^{m_1} \times_\kk \cdots \times_\kk \PP_\kk^{m_p} \rightarrow \PP' = \PP_\kk^{m_{j_1}} \times_\kk \cdots \times_\kk \PP_\kk^{m_{j_k}}$.
	For each  $\bn=(n_1,\ldots,n_p) \in \NN^p$ with $|\bn| = n_1+\cdots+n_p=\dim(X)$,  one says that $\deg_\PP^\bn(X)$ is the \textit{multidegree of $X$ of type $\bn$ with respect to $\PP$} (see \autoref{sect_prelim} for precise definitions).
	Geometrically speaking, when $\kk$ is algebraically closed, $\deg_\PP^\bn(X)$ equals the number of points (counting multiplicity) in the intersection of $X$ with the product $L_1 \times_\kk \cdots \times_\kk L_p  \subset \PP$, where $L_i \subset \PP_\kk^{m_i}$ is a general linear subspace of dimension $m_i-n_i$.
	Let $\text{MDeg}_\PP(X)$ be the maximum of the multidegrees of $X \subset \PP$.
	The next theorem shows that $\text{MDeg}$ cannot increase under the natural projections.
	
	\begin{headthm}[\autoref{thm_mdeg_proj}]
		\label{thmC}
		We have the inequality 
		$$
		{\rm MDeg}_{\PP'}\left(\Pi_\fJ(X)\right) \;\le\; {\rm MDeg}_\PP(X).
		$$
		Moreover, for any $\dd \in \NN^k$ with $|\dd| = \dim\left(\Pi_\fJ(X)\right)$ and such that $\deg_{\PP'}^\dd\left(\Pi_\fJ(X)\right) > 0$, there exists some $\bn \in \NN^p$ with $|\bn| = \dim(X)$ and such that $\deg_\PP^\bn(X) >0$ and $\deg_{\PP'}^\dd\left(\Pi_\fJ(X)\right)$ divides $\deg_\PP^\bn(X)$.
	\end{headthm}

	 If we drop the condition that $P$ is a prime ideal, then one can find fairly simple instances where multidegrees do increase under projections (see \autoref{examp_grow_mdeg_proj}).
	 
	 \smallskip
	 
	 Finally, we also provide an alternative proof of the following remarkable result of Brion \cite{BRION_MULT_FREE} regarding multiplicity-free varieties. 
	 We say that $X \subset \PP$ is \emph{multiplicity-free} if $\deg_\PP^{\bn}(X) \in \{0,1\}$ for all $\bn \in \NN^p$ with $|\bn| = \dim(X)$.
	 Let $\msupp_\PP(X) = \lbrace \bn \in \NN^p \mid |\bn| = \dim(X) \text{ and } \deg_{\PP}^\bn(X) > 0\rbrace$ 
	 be the support of positive multidegrees of $X$.
	 
	 \begin{headthm}[{\autoref{thm_Brion_mult_free}; Brion \cite{BRION_MULT_FREE}}]
	 		\label{thmD}
			If $X \subset \PP = \PP_{\kk}^{m_1} \times_\kk \cdots \times_\kk \PP_{\kk}^{m_p}$ is multiplicity-free, then: 
			\begin{enumerate}[\rm (i)]
				\item $X$ is arithmetically Cohen-Macaulay {\rm(}i.e., $S/P$ is a Cohen-Macaulay ring{\rm)}.
				\item $X$ is arithmetically normal {\rm(}i.e., $S/P$ is a normal domain{\rm)}.
				\item {\rm(}$\kk$ infinite{\rm)} There is flat degeneration of $X$ to the following reduced union of multiprojective spaces
				$$
				H \;=\; \bigcup_{\bn = (n_1,\ldots,n_p) \,\in\, \msupp_\PP(X)} 	\PP_\kk^{n_1} \times_\kk \cdots \times_\kk \PP_\kk^{n_p}  \;\, \subset \;\, \PP = \PP_{\kk}^{m_1} \times_\kk \cdots \times_\kk \PP_{\kk}^{m_p}
				$$		
			\end{enumerate}
			where $\PP_\kk^{n_i} = \Proj\left(\kk[x_{i,m_i-n_i}, \ldots,x_{i,m_i}]\right) \subset \PP_{\kk}^{m_i} = \Proj\left(\kk[x_{i,0}, \ldots,x_{i,m_i}]\right)$ only uses the last $n_i+1$ coordinates.
	 \end{headthm}
	 
	 Our proof of \autoref{thmD} uses techniques quite different from the ones used in \cite{BRION_MULT_FREE}, and ideas related to the fiber-full scheme (\cite{LOC_FIB_FULL_SCHEME, FIBER_FULL,FIB_FULL_SCHEME}) play an important role.
	 We also point out that our proof is valid for arbitrary fields.

	\subsection{Non-standard gradings}\hfill

	Multidegrees in non-standard multigradings are considerably more complicated as they lack a clear multiprojective geometrical content, but they are still important because of the flexibility they provide in terms of applications.
	One prime example is Knutson--Miller's result \cite[Theorem A]{KNUTSON_MILLER_SCHUBERT} expressing double Schubert polynomials and double Grothendieck polynomials as the multidegree polynomials and $\mathcal{K}$-polynomials, respectively, of matrix Schubert varieties with a fine grading.
	This fine grading takes into account both row and column position; 
	we provide some examples that explore the technicalities of this fine grading (up to certain sign changes) in \autoref{sect_examp_det}.
	These fine gradings play a key role in the recent proof by Castillo--Cid-Ruiz--Mohammadi--Monta\~no \cite{DOUBLE_SCHUBERT} of a conjecture of Monical--Tokcan--Yong \cite{monical2019newton} stating the saturated Newton polytope property of double Schubert polynomials.
	
	\smallskip
	
	Let $R = \kk[x_1,\ldots,x_n]$ be a positively $\NN^p$-graded polynomial ring (that is, $\deg(x_i) \in \NN^p \setminus \{\mathbf{0}\}$ for all $1 \le i \le n$ and $\deg(\alpha) = \mathbf{0} \in \NN^p$ for all $\alpha \in \kk$).
	Let $I \subset R$ be an $R$-homogeneous ideal.
	The \emph{multidegree polynomial} $\mathcal{C}(R/I;\ttt) = \mathcal{C}(R/I;t_1,\ldots,t_p) $ of $R/I$  is defined in terms of the multigraded Hilbert series of $R/I$ (see \autoref{sect_prelim} for details).
	In the standard multigraded setting it may be seen as an algebraic encoding of the multidegrees of the corresponding multiprojective scheme (see \autoref{thm_two_mdegs}).
	
	Our approach is to associate to $I$ an ideal $J$ in a standard $\NN^p$-graded polynomial ring $S$ such that $\mathcal{C}(R/I;\ttt) = \mathcal{C}(S/J;\ttt)$.
	We call $J$ the \emph{standardization} of $I$ (see \autoref{setup_standardization}).
	This construction is inspired by the techniques \emph{step-by-step homogenization} introduced by McCullough--Peeva in \cite{McCULLOUGH_PEEVA} for singly-graded settings and \emph{standardization} introduced in \cite{DOUBLE_SCHUBERT} for certain multigraded settings.
	It turns out that, there is a tight relation between $I \subset R$ and its standardization $J \subset S$, as it may be witnessed from the following statements:
	\begin{enumerate}[\rm (i)]
		\item $\codim(I) = \codim(J)$.
		\item $I \subset R$ and $J \subset S$ have the same $\NN^p$-graded Betti numbers. 
		\item $\mathcal{K}(R/I;\ttt) = \mathcal{K}(S/J;\ttt)$ and  $\mathcal{C}(R/I;\ttt) = \mathcal{C}(S/J;\ttt)$.
		\item $R/I$ is a Cohen-Macaulay ring if and only if $S/J$ is a Cohen-Macaulay ring.
		\item Initial ideals are compatible with standardization.
		\item If $I \subset R$ is a prime ideal and it does not contain any variable, then $J \subset S$ is also a prime ideal.
	\end{enumerate}
	For more details, see \autoref{thm_std}.
	
	These results allow us to deduce in \autoref{thm_polymatroid_multdeg} that the support of $\mathcal{C}(R/P;\ttt)$ is a discrete polymatroid for any prime ideal $P \subset R$, which constitutes an extension of \cite[Theorem A]{POSITIVITY}.

	\smallskip
	
	We introduce the notion of \emph{arithmetic multidegree polynomial}, which  can be seen as a generalization of the notion of arithmetic degree considered by Sturmfels--Trung--Vogel in \cite{STV_DEGREE}.
	It is defined as follows
	$$
	\mathcal{A}(R/I; \ttt) \;=\;  \sum_{P \in \Ass_R(R/I)} \text{length}_{R_P}\left(\HH_P^0((R/I)_P)\right) \, \mathcal{C}(R/P; \ttt).
	$$
	Motivated by the work of Hartshorne on the connectivity of Hilbert schemes \cite{HARTSHORNE_CONNNECT}, we show in \autoref{thm_Amdeg_degen} that arithmetic multidegree is upper semicontinuous under flat degenerations. 
	As a consequence, in the case of Gr\"obner degenerations, we obtain the coefficient-wise inequality 
	$$
	\mathcal{A}\left(R/I; \ttt\right) \,\le_c\, \mathcal{A}\left(R/\init_>(I); \ttt\right).
	$$
	This upper semicontinuity result was obtained by Hartshorne \cite[Chapter 2]{HARTSHORNE_CONNNECT} in the singly-graded setting.

	\smallskip
	
	In \autoref{subsect_CS_ideals}, we extend the notion of Cartwright--Sturmfels ideals to arbitrary positive multigradings. 
	The original definition was given over a standard multigraded setting by Conca--De Negri--Gorla in a series of papers \cite{CDNG_CS_IDEALS,CDNG_GIN,CDNG_GRAPH,CDNG_MINORS,conca2022radical}.
	This was obtained as an abstraction of previous work of Cartwright and Sturmfels \cite{CS_PAPER}; thus the chosen name.

	In the current arbitrary positive multigraded setting, we say that $I \subset R$ is \emph{Cartwright--Sturmfels} (CS for short) if there is a radical Borel-fixed ideal in the standard $\NN^p$-graded polynomial ring $S$ that has the same $\mathcal{K}$-polynomial as $I$ (see \autoref{def_new_CS}).
	It should be mentioned that $I \subset R$ is CS if and only if its standardization $J \subset S$ is CS (in the sense of \cite{CDNG_CS_IDEALS}).
	By utilizing the close relation between $I \subset R$ and its standardization, in \autoref{thm_props_CS}, we prove the following statements when $I$ is CS:
	\begin{enumerate}[\rm (i)]
		\item $\init_{>}(I)$ is radical and {\rm CS} for any monomial order $>$ on $R$; in particular, $I$ is radical.
		\item The $\NN$-graded Castelnuovo-Mumford regularity $\reg(I)$ is bounded from above by $p$.
		\item $I$ is a multiplicity-free ideal.
		\item If $P \subset R$ is a minimal prime of $I$, then $P$ is CS.
		\item All reduced Gr\"obner bases of $I$ consist of elements of multidegree $\le (1,\ldots,1) \in \NN^p$.
		In particular, $I$ has a universal Gr\"obner basis of elements of multidegree $\le (1,\ldots,1) \in \NN^p$.
	\end{enumerate}
	This shows that CS ideals in arbitrary positive multigradings enjoy similar desirable properties as in the standard multigraded setting.
	
	\smallskip
	
	Finally, as a consequence of our work, we portray a family of multigraded Hilbert functions whose ideals have very rigid properties. 
	It is most convenient to enunciate this result in terms of the multigraded Hilbert scheme of Haiman and Sturmfels \cite{HAIMAN_STURMFELS}.
	Let $h:\NN^p \rightarrow \NN$ be a function and consider the corresponding multigraded Hilbert scheme $\HS_{R/\kk}^h$ parametrizing $R$-homogeneous ideals with Hilbert function equal to $h$.
	
	\begin{headthm}[\autoref{cor_Hilb_sch}]
		\label{thmE}
		If $\HS_{R/\kk}^h$ contains a $\kk$-point that corresponds to a {\rm CS} ideal and a $\kk$-point that corresponds to a prime ideal, then any $\kk$-point in $\HS_{R/\kk}^h$ corresponds to an ideal that is Cohen-Macaulay and {\rm CS}.
	\end{headthm}
	
	This theorem can be seen as a generalization of a previous result by Cartwright and Sturmfels \cite[Theorem 2.1 and Corollary 2.6]{CS_PAPER}.
	
	\subsection*{Outline}
	The basic outline of this paper is as follows. 
	In \autoref{sect_prelim}, we set the notation and recall basic definitions.
	In \autoref{sect_Hartshorne}, we prove the upper semicontinuity of arithmetic multidegree under flat degenerations.
	\autoref{sect_projections} is dedicated to prove \autoref{thmC}.
	We establish \autoref{thmA} and \autoref{thmB} in \autoref{sect_multdeg_primes}.
	We study multiplicity-free varieties and prove \autoref{thmD} in \autoref{sect_mult_free}.
	In \autoref{sec:standardization}, we develop the new notion of CS ideals and prove \autoref{thmE}.
	Finally, in \autoref{sect_examp_det}, we present several examples of determinantal ideals with a fine grading.

	\section{Preliminaries}
	\label{sect_prelim}
	
	In this preparatory section, we briefly recall some basic concepts and we fix the notation.
	In particular, we establish the initial stage regarding the main theme of this paper, that of multidegrees.
	For more details on multidegrees the reader is referred to \cite{miller2005combinatorial, cidruiz2021mixed}.

	We use a multi-index notation.
	Let $p \ge 1$ be a positive integer and, for each $1 \le i \le p$, let $\ee_i \in \NN^p$ be the $i$-th elementary vector $\ee_i=\left(0,\ldots,1,\ldots,0\right)$.
	If $\bn = (n_1,\ldots,n_p),\, \bm = (m_1,\ldots,m_p) \in \ZZ^p$ are two vectors, we write $\bn \ge \bm$ whenever $n_i \ge m_i$ for all $1 \le i \le p$, and $\bn > \bm$ whenever $n_j > m_j$ for all $1 \le j \le p$.
	For any $\bn =(n_1,\ldots,n_p) \in \NN^p$, we set $\bn! = n_1!\cdots n_p!$.
	We write $\mathbf{0} \in \NN^p$ and $\mathbf{1} \in \NN^p$ for the  vectors $\mathbf{0}=(0,\ldots,0)$ and $\mathbf{1}=(1,\ldots,1)$, respectively.  
	
	A \textit{discrete polymatroid} $\mathcal{P}$ on $[p]:=\{1,\ldots,p\}$ is a collection of points in $\NN^p$ of the following form
	\[
	\mathcal{P=}\left\{(n_1,\ldots, n_p)\in \NN^p \;\mid\; \sum_{j\in\fJ} n_j\leq r(\fJ), \;\forall \fJ\subsetneq [p], \;\sum_{i\in [p]} n_i=r([p])\right\}
	\]
	with $r : 2^{[p]}\rightarrow \NN$ 
	being a rank function on $[p]$. A \emph{rank function on} $[p]$ 
	is a function $r : 2^{[p]}\rightarrow \NN$ satisfying the following three properties: 
	\begin{enumerate}[(i)]
		\item $r(\emptyset) = 0$.
		\item $r(\fJ_1)\leq r(\fJ_2)$ if $\fJ_1\subseteq \fJ_2 \subseteq [p]$.
		\item $r(\fJ_1\cap \fJ_2)+r(\fJ_1\cup \fJ_2)\leq r(\fJ_1)+r(\fJ_2)$ if $\fJ_1,\fJ_2 \subseteq [p]$.
	\end{enumerate}
	
	A comprehensive discussion regarding polymatroids can be found in \cite{schrijver2003combinatorial}.

	\begin{remark}
		\label{rem_basic_polymatroids}
		The following statements hold:
		\begin{enumerate}[\rm (i)]
			\item Let $\mathcal{P}$ be a discrete polymatroid on $[p]$ with rank function $r : 2^{[p]} \rightarrow \NN$.
			Let $m_1,\ldots,m_p$ be positive integers such that $r(\{i\}) \le m_i$.
			Then the function $s : 2^{[p]} \rightarrow \NN$ given by 
			$$
			s(\fJ) \; := \; \sum_{j \in \fJ} m_j + r([p] \setminus \fJ) - r([p]) 
			$$
			is a rank function. 
			The discrete polymatroid $\mathcal{P}^*$ on $[p]$ determined by $s$ is said to be a dual of $\mathcal{P}$.
			\item If $\mathcal{P}_1$ and $\mathcal{P}_2$ are two discrete polymatroids on $[p]$, then the Minkowski sum $\mathcal{P}_1 + \mathcal{P}_2$ is a discrete polymatroid on $[p]$.
		\end{enumerate}
	\end{remark}
	\begin{proof}
		(i) See \cite[\S 44.6f]{schrijver2003combinatorial}.
		
		(ii) See \cite[Theorem 44.6, Corollary 46.2c]{schrijver2003combinatorial}.
	\end{proof}

	Following \cite{monical2019newton}, we say that a polynomial $f = \sum_{\bn}c_{\bn}\ttt^\bn\in \ZZ[t_1,\ldots,t_p]$ has the \emph{Saturated Newton Polytope property} (SNP property for short) if the support  $\supp(f)=\{\bn\in\NN^p\mid c_\bn \neq 0\}$ of $f$ is equal to $\text{Newton}(f)\cap\NN^p$, where $\text{Newton}(f) = \text{ConvexHull}\{\bn\in\NN^n\mid c_\bn \neq 0\}$ denotes the \emph{Newton polytope} of $f$; in other words, if the support of $f$ consists of the integer points of a polytope.

	\begin{remark}
		Given a homogeneous polynomial $f \in \ZZ[t_1,\ldots,t_p]$, if $\supp(f)$ is a discrete polymatroid on $[p]$, then $f$ has the SNP property.
	\end{remark}

	\subsection{Multidegrees in positive multigradings}
	Here we briefly review the notion of multidegrees in positive multigradings that are not necessarily standard.

	Let $\kk$ be a field and $R = \kk[x_1,\ldots,x_n]$ be a \emph{positively $\NN^p$-graded} polynomial ring (that is, $\deg(x_i) \in \NN^p \setminus \{\mathbf{0}\}$ for all $1 \le i \le n$ and $\deg(\alpha) = \mathbf{0} \in \NN^p$ for all $\alpha \in \kk$).  
	Let $M$ be a finitely generated $\ZZ^p$-graded $R$-module and $F_\bullet$ be a $\ZZ^p$-graded free $R$-resolution 
	$
	F_\bullet : \; \cdots \rightarrow F_i \rightarrow F_{i-1} \rightarrow \cdots \rightarrow F_1 \rightarrow F_0
	$
	of $M$.
	Let $t_1,\ldots,t_p$ be variables over $\ZZ$ and consider the polynomial ring $\ZZ[\ttt] = \ZZ[t_1,\ldots,t_p]$, where the variable $t_i$ corresponds with the $i$-th elementary vector $\ee_i \in \NN^p$. 
	If we write $F_i = \bigoplus_{j} R(-\mathbf{b}_{i,j})$ with $\mathbf{b}_{i,j} = (\mathbf{b}_{i,j,1},\ldots,\mathbf{b}_{i,j,p}) \in \ZZ^p$, then we define the Laurent polynomial 
	$
	\left[F_i\right]_\ttt \, := \, \sum_{j} \ttt^{\mathbf{b}_{i,j}} = \sum_{j} t_1^{\mathbf{b}_{i,j,1}} \cdots t_p^{\mathbf{b}_{i,j,p}}.
	$
	Then, the \emph{$\mathcal{K}$-polynomial} of $M$ is defined by 
	$$
	\mathcal{K}(M;\ttt) \, := \, \sum_{i} {(-1)}^i \left[ F_i \right]_\ttt.
	$$
	\begin{remark}
		Since $R$ is assumed to be positively graded, there is a well-defined notion of Hilbert series 
		$$
		\Hilb_M(\ttt) \;:= \; \sum_{\bn \in \ZZ^p} \dim_\kk\left(\left[M\right]_\bn\right) \ttt^\bn,
		$$
		and then we can write
		$$
		\Hilb_M(\ttt) \;= \;  \frac{\mathcal{K}(M;\ttt)}{\prod_{i=1}^n \left(1 - \ttt^{\deg(x_i)}\right)}.
		$$
		In particular, this shows that $\mathcal{K}(M;\ttt)$ does not depend on the specific free resolution $F_\bullet$
	(also, see \cite[Theorem 8.34]{miller2005combinatorial}).
	\end{remark}

	\begin{definition}
		\label{def_multdeg_pol}
		The \emph{multidegree polynomial} of a finitely generated $\ZZ^p$-graded $R$-module $M$ is the homogeneous polynomial $\mathcal{C}(M; \ttt) \in \ZZ[\ttt]$ given as the sum of all terms in 
		$$
		\mathcal{K}(M; \mathbf{1} - \ttt) = \mathcal{K}(M; 1-t_1,\ldots,1-t_p) 
		$$
		having total degree $\codim(M)$, which is the lowest degree appearing.
	\end{definition}
	
	An $R$-homogeneous ideal $I \subset R$ is said to be \emph{multiplicity-free} if the coefficients of $\mathcal{C}(R/I;\ttt)$ belong to the set $\{0, 1\}$.
	
	We also have the following variation of the multidegree polynomial that was introduced in \cite{CDNG_MINORS}.
	This notion was introduced with the goal of measuring the contribution of all the minimal primes (possibly of not maximal dimension), and in the case of Borel-fixed ideals, it does precisely that (see \cite[Proposition 3.12]{CDNG_MINORS}).
	
	\begin{definition}
		\label{def_G_multdeg}
		The \emph{$\mathcal{G}$-multidegree polynomial} of a finitely generated $\ZZ^p$-graded $R$-module $M$ is the  polynomial $\mathcal{G}(M; \ttt) \in \ZZ[\ttt]$ given as the sum of all terms in 
		$
		\mathcal{K}(M; \mathbf{1} - \ttt)
		$
		whose corresponding monomial is minimal in the support of $\mathcal{K}(M; \mathbf{1} - \ttt)$.
	\end{definition}

	For a polynomial $p \in \ZZ[\ttt]$, we denote by $\left[p\right]_i$ the sum of the terms of $p$ with total degree equal to $i$. 
	So, for any finitely generated $\ZZ^p$-graded module $M$ with $\codim(M) \ge i$, we obtain that $\left[\mathcal{K}(M; \mathbf{1} - \ttt)\right]_i$ equals $\mathcal{C}(M;\ttt)$ if $\codim(M) = i$ and $0 \in \ZZ[\ttt]$ otherwise.
	We collect the following basic facts.

	\begin{remark}
		\label{rem_multdeg_facts}
		Let $M$ be a finitely generated $\ZZ^p$-graded $R$-module.
		The following statements hold: 
		\begin{enumerate}[\rm (i)]
			\item $\mathcal{C}(M; \ttt)$ is the sum of the terms of $\mathcal{G}(M;\ttt)$ having total degree equal to $\codim(M)$.
			\item If $z \in R$ is a homogeneous non-zero-divisor on $M$ with $\deg(z) = (a_1,\ldots,a_p) \in \NN^p$, then 
			\begin{enumerate}[\rm (a)]
				\item $\mathcal{K}(M/zM;\ttt) = (1 - t_1^{a_1}\cdots t_p^{a_p}) \, \mathcal{K}(M; \ttt)$,
				\item $\mathcal{C}(M/zM;\ttt) = \langle \deg(z), \ttt \rangle \, \mathcal{C}(M; \ttt)$ where $\langle \deg(z), \ttt \rangle = a_1t_1+ \cdots + a_pt_p$.
			\end{enumerate}			
			\item[(iii)\,(additivity)]  Let $0 \rightarrow L \rightarrow M \rightarrow N \rightarrow 0$ be a short exact sequence of finitely generated $\ZZ^p$-graded $R$-modules.
			Then we have the equality $\mathcal{C}(M;\ttt) = \left[\mathcal{K}(L; \mathbf{1} - \ttt)\right]_i + \left[\mathcal{K}(N; \mathbf{1} - \ttt)\right]_i$ with $i = \codim(M)$.
			\item[(iv)\,(associativity formula)]  $$\mathcal{C}(M; \ttt) =  \sum_{\substack{P \in \Ass_R(M)\\
			\dim(R/P) = \dim(M)}} \text{length}_{R_P}\left(M_P\right) \, \mathcal{C}(R/P; \ttt).
			$$
			\item[(v)\,(positivity)] $\mathcal{C}(M;\ttt)$ is a nonzero polynomial with non-negative coefficients.
		\end{enumerate}
	\end{remark}
	\begin{proof}
		(i) This part follows from the fact that any term in $\mathcal{K}(M; \mathbf{1} - \ttt)$ has total degree at least $\codim(M)$ (see \cite[Claim 8.54]{miller2005combinatorial}).
		
		(ii) It can be obtained by simple algebraic manipulations with Hilbert series (see \cite[Exercise 8.12]{miller2005combinatorial}).
		
		(iii) It is a direct consequence of the additivity of Hilbert series and the inequalities $\codim(L) \ge \codim(M)$ and $\codim(N) \ge \codim(M)$.
		
		(iv) This is proved by a standard use of prime filtrations; see \cite[Theorem 8.53]{miller2005combinatorial}.
		
		(v) Write $M \cong F/K$ with $F$ a finite rank $\ZZ^p$-graded free $R$-module.
		By considering a Gr\"obner degeneration we may assume that $K$ is a monomial submodule (see, e.g., \cite[Chapter 15]{EISEN_COMM}).
		Then the associativity formula (part (iv)) reduces to the case $M = R/P$ and $P$ a monomial prime ideal.
		For a monomial prime ideal $P = (x_{i_1},\ldots,x_{i_m})$, part (ii)(b) yields $\mathcal{C}(R/P;\ttt) = \langle \deg(x_{i_1}), \ttt \rangle \cdots \langle \deg(x_{i_m}), \ttt \rangle$, and so the result follows from the fact that $R$ is positively $\NN^p$-graded.
		Alternatively, see \autoref{rem_posit_non_std}.
	\end{proof}

	We introduce yet another notion of multidegree polynomial. 
	It can be seen as a generalization of the notion of arithmetic degree from \cite{STV_DEGREE} (also, see \cite{HARTSHORNE_CONNNECT}).
	The goal of this notion is to capture the contribution of all the associated primes.
	
	\begin{definition}
		\label{def_A_multdeg}
		The \emph{arithmetic multidegree polynomial} of a finitely generated $\ZZ^p$-graded $R$-module $M$ is given by 
		$$
		\mathcal{A}(M; \ttt) \;:=\;  \sum_{P \in \Ass_R(M)} \text{length}_{R_P}\left(\HH_P^0(M_P)\right) \, \mathcal{C}(R/P; \ttt).
		$$
	\end{definition}

	\subsection{Multidegrees in standard multigradings}
	
	We now concentrate on standard multigradings.
	The study of standard multigraded algebras is of utmost importance as they correspond with closed subschemes of a product of projective spaces (see, e.g., \cite{POSITIVITY} and the references therein).
	
	Let $S = \kk[x_1,\ldots,x_n]$ be a \emph{standard $\NN^p$-graded} polynomial ring over a field $\kk$. 
	That is, the total degree of each variable $x_i$ is equal to one (i.e.,~for each $1 \le i \le n$, we have $\deg(x_i) = \ee_{k_i} \in \NN^p$ with $1 \le k_i \le p$) and $\deg(\alpha) = \mathbf{0} \in \NN^p$ for all $\alpha \in \kk$.

	Suppose that for each $1 \le i \le p$ there are exactly $m_i+1$ variables $x_j$'s with degree $\ee_i$ (i.e., $m_i +1 = \lvert \{ 1 \le j \le n \mid \deg(x_j) = \ee_i \} \rvert$).
	Thus the corresponding multiprojective scheme is the product of projective spaces $\multProj(S) = \PP_\kk^{m_1} \times_\kk \cdots \times_\kk \PP_\kk^{m_p}$.
	In this setting, the irrelevant ideal is given by $\nn := \left([S]_{\ee_1}\right) \cap \cdots \cap \left([S]_{\ee_p}\right) \subset S$.
	An $S$-homogeneous prime ideal $P \subset S$ is said to be \emph{relevant} if $P \not\supseteq \nn$.
	Relevant prime ideals are important as they are the ones that have a well-defined geometrical counterpart.
	For more details on the $\multProj$ construction, the reader is referred to \cite{POSITIVITY,HYRY_MULTGRAD}.
	The next remark shows that one can always restrict to relevant prime ideals when considering the multidegree polynomial.

	\begin{remark}
		\label{rem_relev_prim}
		Let $S' = S[x_{n+1},\ldots,x_{n+p}]$ be a standard $\NN^p$-graded polynomial ring extending the grading of $S$ and with $\deg(x_{n+i}) = \ee_i$ for all $1 \le i \le p$; accordingly, we have that $\multProj(S')  = \PP_\kk^{m_1+1} \times_\kk \cdots \times_\kk \PP_\kk^{m_p+1}$.
		Let $P \subset S$ be an $S$-homogeneous (not necessarily relevant) prime ideal, and set $P' = PS'$ to be the extension of $P$ to $S'$.
		Then the following statements hold: 
		\begin{enumerate}[\rm (i)]
			\item $P' \subset S'$ is a relevant prime ideal.
			\item $\mathcal{C}(S/P;\ttt) = \mathcal{C}(S'/P'; \ttt)$.
		\end{enumerate}
	\end{remark}
	\begin{proof}
		(i) It is clear that $P'$ is a prime ideal, and it is relevant because $x_{n+i} \notin P'$ for all $1 \le i \le p$.
		
		(ii) Given a $\ZZ^p$-graded free $S$-resolution $F_\bullet$ of $S/P$, then $F_\bullet \otimes_S S'$ is a $\ZZ^p$-graded free $S'$-resolution of $S'/P'$ with the same Betti numbers. 
		Therefore, $\mathcal{K}(S/P;\ttt) = \mathcal{K}(S'/P'; \ttt)$ and consequently $\mathcal{C}(S/P;\ttt) = \mathcal{C}(S'/P'; \ttt)$.
	\end{proof}

	Let $M$ be a finitely generated $\ZZ^p$-graded $S$-module, and set $\PP := \multProj(S) = \PP_\kk^{m_1} \times_\kk \cdots \times_\kk \PP_\kk^{m_p}$.
	The relevant support of $M$ is given by $\Supp_{++}(M) := \Supp(M) \cap \PP$.
	There is a polynomial $P_M(\ttt)=P_M(t_1,\ldots,t_p) \in \QQ[\ttt]=\QQ[t_1,\ldots,t_p]$, called the \emph{Hilbert polynomial} of $M$ (see, e.g., \cite[Theorem 4.1]{HERMANN_MULTIGRAD}, \cite[Theorem 3.4]{cidruiz2021mixed}), such that the degree of $P_M$ is equal to $r = \dim\left(\Supp_{++}(M)\right)$ and 
	$$
	P_M(\nu) = \dim_\kk\left([M]_\nu\right) 
	$$
	for all $\nu \in \NN^p$ such that $\nu \gg \mathbf{0}$.
	Furthermore, if we write 
	\begin{equation}
		\label{eq_Hilb_poly}
		P_{M}(\ttt) = \sum_{n_1,\ldots,n_p \ge 0} e(n_1,\ldots,n_p)\binom{t_1+n_1}{n_1}\cdots \binom{t_p+n_p}{n_p},
	\end{equation}
	then $0 \le e(n_1,\ldots,n_p) \in \ZZ$ for all $n_1+\cdots+n_p = r$.

	\begin{definition} \label{def_multdeg}
		Under the notation of \autoref{eq_Hilb_poly}, we define the following invariants: 
		\begin{enumerate}[(i)]
			\item For $\bn = (n_1,\ldots,n_p) \in \NN^p$ with $\lvert\bn\rvert=\dim\left(\Supp_{++}(M)\right)$, $e(\bn;M) := e(n_1,\ldots,n_p)$ is the \textit{mixed multiplicity of $M$ of type $\mathbf{n}$}.
			
			\item Let $R$ be an $\NN^p$-graded quotient ring of $S$ and
			 $X = \multProj(R) \subset \PP$ be the corresponding closed subscheme.
			For each $\bn  \in \NN^p$ with $\lvert\bn\rvert=\dim(X)$, $\deg_\PP^\bn(X):=e(\bn; R)$ is the \textit{multidegree of $X$ of type $\bn$ with respect to $\PP$}. 
		\end{enumerate} 	\end{definition}
	
		\begin{definition}
		For a multiprojective scheme $X \subset \PP = \PP_\kk^{m_1} \times_\kk \cdots \times_\kk \PP_\kk^{m_p}$, the support of positive multidegrees of $X$ is denoted by $\msupp_\PP(X) := \lbrace \bn \in \NN^p \mid |\bn| = \dim(X) \text{ and } \deg_{\PP}^\bn(X) > 0\rbrace$.
	\end{definition}
	
	\begin{remark}
		In a multigraded setting one needs to be careful because the notions of multidegrees introduced in \autoref{def_multdeg_pol} and \autoref{def_multdeg} may not agree.
		This pathology stems from the fact that the support $\Supp(M)$ and the relevant support $\Supp_{++}(M)$ may be quite different for a $\ZZ^p$-graded $S$-module $M$.
		Nevertheless, we have the following unifying result.
	\end{remark}

	\begin{theorem}[{\cite{cidruiz2021mixed}}]
		\label{thm_two_mdegs}
		Let $M$ be a finitely generated $\ZZ^p$-graded $S$-module and set $r = \dim\left(\Supp_{++}(M)\right)$. 
		If $\HH_{\nn}^0(M) = 0$, then we have equality 
		$$
		\mathcal{C}(M;\ttt) \; = \; \sum_{\substack{\bn \in \NN^p\\ |\bn| = r}} e(\bn; M) \, t_1^{m_1-n_1} \cdots t_p^{m_p-n_p}.
		$$
	\end{theorem}
	\begin{proof}
		This is a consequence of \cite[Theorem A]{cidruiz2021mixed}.
		Also, for more details, see \cite[Remark 2.9]{POSITIVITY}.
	\end{proof}

	For a closed subscheme $X \subset \PP = \PP_\kk^{m_1} \times_\kk \cdots \times_\kk \PP_\kk^{m_p}$, we say that the \emph{multidegree polynomial of $X$} is defined as
	$$
	\mathcal{C}(X;\ttt) \; := \; \sum_{\substack{\bn \in \NN^p\\ |\bn| = \dim(X)}} \deg_\PP^\bn(X) \, t_1^{m_1-n_1} \cdots t_p^{m_p-n_p}.
	$$
	The above theorem implies that $\mathcal{C}(X;\ttt) = \mathcal{C}(R;\ttt)$ for an $\NN^p$-graded quotient ring $R$ of $S$ such that $X = \multProj(R)$  and $(0:_R \nn^\infty) = 0$.

	The following piece of notation will be useful throughout the paper.	

	\begin{notation}[Natural projections]
		\label{nota_projections}
		Let $\mathfrak{J} = \{j_1,\ldots,j_k\} \subseteq [p] = \{1, \ldots, p\}$ be a subset. 
		\begin{enumerate}[\rm (i)]
			\item Let $\Pi_\fJ : \PP_\kk^{m_1} \times_\kk \cdots \times_\kk \PP_\kk^{m_p} \rightarrow \PP_\kk^{m_{j_1}} \times_\kk \cdots \times_\kk \PP_\kk^{m_{j_k}}$ denote the natural projection.
			\item Denote by $S_{(\fJ)}$ the $\NN^k$-graded $\kk$-algebra given by 
			$$
			S_{(\fJ)} := \bigoplus_{\substack{i_1\ge 0,\ldots, i_p\ge 0\\ i_{j} = 0 \text{ if } j \not\in \fJ}} {\left[S\right]}_{(i_1,\ldots,i_p)},
			$$
			and denote by $I_{(\fJ)}$ the contraction $I_{(\fJ)} := I \cap S_{(\fJ)}$ for any $S$-homogeneous ideal $I \subset S$.		
		\end{enumerate}
	\end{notation}
	
		Since the coefficients of the multidegree polynomial are non-negative in positive gradings, it becomes natural to address the positivity of these coefficients. 
	The next result characterizes the positivity of multidegrees, and it follows from \cite{POSITIVITY}.
	
	\begin{theorem}[\cite{POSITIVITY}]
		\label{thm_pos_multdeg}
		Let $P \subset S$ be an $S$-homogeneous {\rm(}not necessarily relevant{\rm)} prime ideal.
		Then the support of $\mathcal{C}(S/P; \ttt)$ is a discrete  polymatroid.
	\end{theorem}
	\begin{proof}
		By \autoref{rem_relev_prim}, we may adjoin new variables to $S$ and assume that $P \subset S$ is a relevant prime ideal.
		Consider the closed subscheme $X = \multProj(S/P) \subset \PP = \PP_\kk^{m_1}\times_\kk \cdots \times_\kk \PP_\kk^{m_p}$.
		Then \autoref{thm_two_mdegs} implies that
		$
		\mathcal{C}(S/P;\ttt) \,=\, \sum_{\substack{\bn \in \NN^p \\ |\bn| = \codim(P)}} \deg_\PP^{\bm-\bn}(X) \ttt^\bn \; \in \; \NN[t_1,\ldots, t_p].
		$

		Fix $\bn \in \NN^p$ with $|\bn| = \codim(P)$.
		From \cite[Theorem A]{POSITIVITY} we obtain that $\deg_\PP^{\bm-\bn}(X) > 0$ if and only if $\sum_{j \in \fJ }m_{j} - \sum_{j \in \fJ }n_{j} \le  \dim\big(\Pi_\fJ(X)\big)$ for all $\fJ \subset [p]$.
		We have that $r : 2^{[p]} \rightarrow \NN, \, r(\fJ) := \dim\big(\Pi_\fJ(X)\big)\big)$ is a rank function (see \cite[Proposition 5.1]{POSITIVITY}).
		Equivalently, we get that $\deg_\PP^{\bm-\bn}(X) > 0$ if and only if
		$$
		\sum_{j \in \fJ} n_j \;\le\; s(\fJ) := \sum_{j \in \fJ} m_j + r([p] \setminus \fJ) - r([p]) \quad \text{ for all \quad $\fJ \subseteq [p]$},
		$$
		where $s : 2^{[p]} \rightarrow \NN$ is a rank function by \autoref{rem_basic_polymatroids}(i).
		Therefore the support of $\mathcal{C}(S/P; \ttt)$ is a discrete  polymatroid, as claimed.
	\end{proof}

\section{A multidegree version of a result of Hartshorne}
\label{sect_Hartshorne}

In this section, we extend some results of Hartshorne into the world of multidegrees. 
More precisely, by extending the results of \cite[Chapter 2]{HARTSHORNE_CONNNECT}, we study the behavior of the arithmetic multidegree under flat degenerations (\autoref{def_A_multdeg}). 
We shall see that the arithmetic multidegree is the multigraded extension of the numbers $n_i$'s introduced by Hartshorne in \cite[Chapter 2]{HARTSHORNE_CONNNECT}.

\subsection{The functor $R^i$}

Here we recall the definition of the functors $R^i$ and some of their basic properties.
Essentially the same construction was  considered by Schenzel \cite{SCHENZEL_CMF}, under the name of \emph{dimension filtration}, in his study of Cohen-Macaulay filtered modules.

Let $R$ be a Noetherian ring and $M$ be a finitely generated $R$-module.

\begin{definition}
		For any $i \ge 0$, we set 
		$$
		R^i(M) \;:= \; \big\lbrace m \in M \mid \codim\left(\Supp_R(m)\right) \ge i  \big\rbrace.
		$$
\end{definition}

We first pause to point out some of the basic properties of the functor $R^i$.

\begin{remark}
	\label{rem_prop_Ri}
	The following statements hold: 
	\begin{enumerate}[\rm (i)]
		\item $R^i(M)$ is an $R$-module.
		\item $R^i$ is a left exact functor.
		\item We have the equality 
		$
		R^i(M) \,=\, \HH_{J^i(M)}^0(M)
		$
		where 
		$$J^i(M) \,:= \bigcap_{\substack{P \in \Ass_R(M)\\ \codim(P) \ge i}} P
		$$ is the intersection of the associated primes of $M$ of codimension $\ge i$.
		By convention, we set $J^i(M) = R$ if there is no $P \in \Ass_R(M)$ with $\codim(P) \ge i$.
		\item Let $0 = \bigcap_j M_j$ be an irredundant primary decomposition of the submodule $0 \subset M$, and suppose that $M_j$ is $P_j$-primary.
		Then $R^i(M)$ is the intersection of those $M_j$ such that $\codim(P_j) < i$.
		\item $\Ass_R\left(R^i(M)\right) = \lbrace P \in \Ass_R(M) \mid \codim(P) \ge i \rbrace$.
		\item  $\Ass_R\left(M/R^i(M)\right) = \lbrace P \in \Ass_R(M) \mid \codim(P) < i \rbrace$.
		\item If $N$ is an $R$-submodule of $M$ supported on codimension $\ge i$, then $N \subset R^i(M)$.
	\end{enumerate}	
\end{remark}
\begin{proof}	
	(i) and (ii) are immediate from the definition of $R^i$.
	
	(iii) It is clear that $R^i(M) \supseteq \HH_{J^i(M)}^0(M)$.
	Let $m \in R^i(M)$.
	By definition, $\codim(\Ann_R(m)) = \codim(\Supp_R(m)) \ge i$, and so it follows that $\Ass_R(R/\Ann_R(m)) \subseteq \lbrace P \in \Ass_R(M) \mid \codim(P) \ge i \rbrace$.
	This implies that $J^i(M) \subseteq \sqrt{\Ann_R(m)}$ and consequently $m \in \HH_{J^i(M)}^0(M)$.
	Thus we have the claimed equality.
	
	(iv), (v) and (vi) follow from the already proved part (iii) and known behavior of the zeroth local cohomology with respect to associated primes and primary decompositions (see \cite[Proposition 3.13]{EISEN_COMM}).
	
	(vii) follows directly from the definition of $R^i$.
\end{proof}

	The proposition below extends \cite[Proposition 2.2]{HARTSHORNE_CONNNECT}.
	
\begin{lemma}
	\label{lem_flat_base_change_Ri}
	Let $\phi : A \rightarrow B$ be homomorphism of Noetherian rings which is either faithfully flat or a localization map with $B = W^{-1}A$ and $W \subset A$ a multiplicatively closed set.
	Let $R$ be a finitely generated $A$-algebra and $M$ be a finitely generated $R$-module.
	Set $R_B = R \otimes_A B$ and $M_B = M \otimes_A B$.
	Then 
	$$
	R^i(M_B) \,\cong\, R^i(M) \otimes_A B 
	$$
	for all $i \ge 0$.
\end{lemma}	
\begin{proof}
		From \autoref{rem_prop_Ri}(iii) and the base change property of local cohomology (in either case $\phi$ is flat), we obtain
		$$
		R^i(M) \otimes_A B \,=\, \HH_{J^i(M)}^0(M) \otimes_A B \,\cong\, \HH_{J^i(M)R_B}^0(M_B).
		$$
		First, we assume that $\phi : A \rightarrow B = W^{-1}B$ is a localization map.
		Then, by basic properties of localizations,  $\Ass_{R_B}(M_B) = \lbrace PR_B \mid P \in  \Ass_R(M) \text{ and } P \not\supset W \rbrace$ and $\codim(PR_B) = \codim(P)$ for any $P \in \Spec(R)$.
		As a consequence, the equality $J^i(M)R_B = J^i(M_B)$ follows, and so we get $R^i(M) \otimes_A B \cong \HH_{J^i(M)R_B}^0(M_B) = \HH_{J^i(M_B)}^0(M_B) = R^i(M_B)$.
		
		Next, we assume that $\phi : A \rightarrow B$ is faithfully flat.
		By \cite[Theorem 23.2]{MATSUMURA} the associated primes of $M_B$ are given by 
		$$
		\Ass_{R_B}(M_B)  \,=\, \bigcup_{P \in \Ass_R(M)} \Ass_{R_B}\left(R_B/PR_B\right).
		$$
		On the other hand, since $R \rightarrow R_B$ is faithfully flat, it follows that $\codim\left(J^i(M)\right)=\codim\left(J^i(M)R_B\right)$ (see \cite[(4.D) Theorem 3, page 28]{MATSUMURA_OLD} and \cite[(13.B) Theorem 19(3), page 79]{MATSUMURA_OLD}).
		This yields the equality $\sqrt{J^i(M)R_B} = \sqrt{J^i(M_B)}$ (see \autoref{rem_prop_Ri}(iii)).
		So, we obtain $R^i(M) \otimes_A B \cong \HH_{J^i(M)R_B}^0(M_B) = \HH_{J^i(M_B)}^0(M_B) = R^i(M_B)$, and the result follows.
\end{proof}

	We now make the connection between the functors $R^i$ and the arithmetic multidegree introduced in \autoref{def_A_multdeg}.

	\begin{defprop}
		Let $\kk$ be a field and $R = \kk[x_1,\ldots,x_n]$ be a positively  $\NN^p$-graded polynomial ring.
		For any $i \ge 0$, we define the \emph{$i$-th truncation multidegree polynomial} by 
		$$
		\mathcal{C}^i(M;\ttt) \,:= \,  \left[\mathcal{C}(R^i(M);\ttt)\right]_i \, = \, \begin{cases}
			\mathcal{C}(R^i(M);\ttt) & \text{ if } \codim(R^i(M)) = i  \\
			0 & \text{ otherwise.}
		\end{cases}
		$$
		Then one has the equality 
		$$
		\mathcal{A}(M; \ttt)  \, = \, \sum_{i \ge 0} \mathcal{C}^i(M; \ttt).
		$$
	\end{defprop}
\begin{proof}
	For any $P \in \Ass_R(M)$ with $\codim(P) = i$, \autoref{rem_prop_Ri}(iii) yields the following equality
	$$
	R^i(M)_P \,=\, \HH_P^0(M_P).
	$$
	Thus by \autoref{rem_prop_Ri}(v) and the associativity formula of multidegrees (see \autoref{rem_multdeg_facts}(iv)), we obtain
	$$
	\left[\mathcal{C}(R^i(M);\ttt)\right]_i \,=\, \sum_{\substack{P \in \Ass_R(M)\\ \codim(P) = i}} \text{length}_{R_P}\left(\HH_P^0(M_P)\right) \mathcal{C}(R/P;\ttt).
	$$
	Hence summing over $i \ge 0$ gives the claimed equality.
\end{proof}

	\subsection{Behavior under flat degenerations}
	
	In this subsection, we study the functor $R^i$ under a flat degeneration.
	The following setup is assumed throughout the present subsection.
	
	\begin{setup}
		\label{setup_flat_degen_Hartshorne}
		Let $A$ be a universally catenary Noetherian domain. 
		Let $R = A[x_1,\ldots,x_n]$ be a positively  $\NN^p$-graded polynomial ring over $A$.
		Let $Q = \Quot(A)$ be the field of fractions of $A$. 
		For any $\pp \in \Spec(A)$, let $\kappa(\pp) := A_\pp/ \pp A_\pp$ be the residue field at $\pp$.
	\end{setup}
	
	Notice that, if $M$ is a finitely generated $\ZZ^p$-graded $R$-module and $\pp \in \Spec(A)$, then $M \otimes_A \kappa(\pp)$ is a finitely generated $\ZZ^p$-graded module over the positively $\NN^p$-graded polynomial ring $R \otimes_A \kappa(\pp) \cong \kappa(\pp)[x_1,\ldots,x_n]$.
	The goal of this subsection is to study the behavior of $\mathcal{A}(M \otimes_A \kappa(\pp); \ttt)$ when varying $\pp \in \Spec(A)$.
	We have the following generic base change property of $R^i$.
	
\begin{proposition}[{\cite[Proposition 2.3]{HARTSHORNE_CONNNECT}}]
	\label{prop_gen_eq_Ri}
	Let $M$ be a finitely generated $\ZZ^p$-graded $R$-module.
	Then there is a dense open subset $\mathcal{U} \subset \Spec(A)$ such that $R^i(M) \otimes_A \kappa(\pp) \cong R^i(M \otimes_A \kappa(\pp))$ for all $i \ge 0$ and $\pp \in \Spec(A)$.
\end{proposition}	

In the polynomial ring $\ZZ[\ttt] = \ZZ[t_1,\ldots,t_p]$ we consider a coefficient-wise order. 
Let $\ge_c$ be the partial order on $\ZZ[\ttt]$ such that for any two polynomials $f = \sum_\bn f_\bn \ttt^\bn$ and $g = \sum_\bn g_\bn \ttt^\bn$ in $\ZZ[\ttt]$ we have $f \ge_c g$ if and only if $f_\bn \ge g_\bn$ for all $\bn \in \NN^p$.
The next proposition shows that the arithmetic multidegree polynomial cannot decrease under flat degenerations. 

\begin{proposition}
	\label{prop_ineq_A_multdeg}
	Let $M$ be a finitely generated $\ZZ^p$-graded $R$-module, and suppose that $M$ is flat over $A$.
	Then we have the coefficient-wise inequality 
	$$
	\mathcal{A}\left(M \otimes_A Q; \ttt\right) \;\le_c\;  \mathcal{A}\left(M \otimes_A \kappa(\pp); \ttt\right)
	$$
	for all $\pp \in \Spec(A)$.
\end{proposition}
\begin{proof}
	Fix $\pp \in \Spec(A)$.
	By \cite[Exercise II.4.11]{HARTSHORNE} or \cite[Proposition 7.1.7]{EGAII} there is a discrete valuation ring $V$ of $Q$ that dominates $A_\pp$; that is, $A_\pp \subset V$ and $\pp A_\pp = \mathfrak{n} \cap A_\pp$ where $\mathfrak{n}$ is the closed point of $\Spec(V)$. 
	We have a field extension $\kappa(\pp) \hookrightarrow \kappa(\mathfrak{n})$ and by utilizing \autoref{lem_flat_base_change_Ri}, we obtain that 
	$$
	R^i\left(M_V \otimes_V \kappa(\mathfrak{n})\right)  \, \cong \, R^i\left(M \otimes_A \kappa(\pp)\right) \otimes_{\kappa(\pp)} \kappa(\mathfrak{n}),
	$$
	where $M _V = M \otimes_{A} V$.
	Since multigraded Hilbert functions are preserved under base field extensions, we get the equality of polynomials $\mathcal{A}\left(M \otimes_A \kappa(\pp); \ttt\right) = \mathcal{A}\left(M_V \otimes_V \kappa(\mathfrak{n}); \ttt\right)$.
	Therefore, we can substitute $A$ by $V$, and for the rest of the proof we assume that $A$ is a discrete valuation ring with maximal ideal $\pp$.
	
	Over the discrete valuation ring $A$, a module is $A$-flat if and only if it is $A$-torsion-free.
	Since $M$ is $A$-flat by assumption, we have $P \cap A = 0$ for all $P \in \Ass_R(M)$.
	Then \autoref{rem_prop_Ri}(v)(vi) imply that both $R^i(M)$ and $M/R^i(M)$ are $A$-flat for all $i \ge 0$.
	From \autoref{lem_flat_base_change_Ri}, we get $R^i(M) \otimes_A Q \cong R^i(M \otimes_A Q)$.
	The $A$-flatness of $R^i(M)$ yields the equalities
	$$
	\mathcal{C}\left(R^i(M\otimes_A Q) ; \ttt\right) \, = \, \mathcal{C}\left(R^i(M) \otimes_A Q; \ttt\right) \, = \, \mathcal{C}\left(R^i(M) \otimes_A \kappa(\pp); \ttt\right),
	$$
	and so we obtain $\codim(R^i(M) \otimes_A \kappa(\pp)) \ge i$.
	The $A$-flatness of $M/R^i(M)$ gives the injection
	$$
	R^i(M) \otimes_A \kappa(\pp) \hookrightarrow M \otimes_A \kappa(\pp).
	$$
	Then we obtain the injection $R^i(M) \otimes_A \kappa(\pp) \hookrightarrow R^i(M \otimes_A \kappa(\pp))$ (see \autoref{rem_prop_Ri}(vii)), and so the additivity and positivity of multidegree polynomials give the inequality 
	$$
	\mathcal{C}^i\left(R^i(M) \otimes_A \kappa(\pp); \ttt\right) \, \le_c \, \mathcal{C}^i\left(R^i(M \otimes_A \kappa(\pp)); \ttt\right)
	$$ 
	(see \autoref{rem_multdeg_facts}(iii)(v)).
	By combining everything together we get the inequality 
	$$
	\mathcal{C}^i\left(M\otimes_A Q; \ttt\right) \,=\, \mathcal{C}^i\left(R^i(M\otimes_A Q); \ttt\right) \, \le_c \, \mathcal{C}^i\left(R^i(M \otimes_A \kappa(\pp)); \ttt\right) \,=\, \mathcal{C}^i\left(M\otimes_A \kappa(\pp); \ttt\right).
	$$
	This concludes the proof of the proposition.
\end{proof}
	
	We are now ready to present the main result of this section.
	
	\begin{theorem}
		\label{thm_Amdeg_degen}
		Assume \autoref{setup_flat_degen_Hartshorne}.
		Let $M$ be a finitely generated $\ZZ^p$-graded $R$-module, and suppose that $M$ is flat over $A$.
		Then the function 
		$$
		\Spec(A) \rightarrow \ZZ[\ttt], \quad \pp \mapsto \mathcal{A}\left(M \otimes_A \kappa(\pp); \ttt\right)
		$$
		is upper semicontinuous with respect to the component-wise order $\ge_c$ on $\ZZ[\ttt]$.
	\end{theorem}
\begin{proof}
		It suffices to show that 
		$$
		E_h \,:=\, \left\lbrace \pp \in \Spec(A) \mid  \mathcal{A}\left(M \otimes_A \kappa(\pp); \ttt\right) \le_c h \right\rbrace
		$$
		is an open subset for an arbitrary $h \in \ZZ[\ttt]$.
		Let $h \in \ZZ[\ttt]$.
		The topological Nagata criterion for openness (see, e.g., \cite[Theorem 24.2]{MATSUMURA}) says that $E_h$ is open if and only if the following two conditions are satisfied:
		\begin{enumerate}[\rm (i)]
			\item If $\pp, \qqq \in \Spec(A)$, $\pp \in E_h$ and $\pp \supseteq \qqq$, then $\qqq \in E_h$.
			\item If $\pp \in E_h$, then $E_h$ contains a non-empty open subset of $V(\pp) \subset \Spec(A)$. 
		\end{enumerate}
		Condition~(i) follows from \autoref{prop_ineq_A_multdeg}.
		Condition~(ii) is obtained by utilizing \autoref{prop_gen_eq_Ri} and Grothendieck's generic freeness lemma (see, e.g., \cite[Theorem 24.1]{MATSUMURA}).
		This establishes the result of the theorem.
\end{proof}

\subsection{Behavior under Gr\"obner degenerations}	

	For the sake of completeness, we restate our results in terms of Gr\"obner degenerations.
	
	\begin{theorem}
		Let $\kk$ be a field, $R = \kk[x_1,\ldots,x_n]$ be a positively $\NN^p$-graded polynomial ring and $>$ be a monomial order {\rm(}or a weight order{\rm)} on $R$.
		For any $R$-homogeneous ideal $I \subset R$, we have that 
		$$
		\mathcal{A}\left(R/I; \ttt\right) \,\le_c\, \mathcal{A}\left(R/\init_>(I); \ttt\right).
		$$
	\end{theorem}
\begin{proof}
	It is a direct consequence of \cite[Theorem 15.17]{EISEN_COMM} and \autoref{prop_ineq_A_multdeg}.
\end{proof}

	This result portrays a different behavior of the arithmetic multidegree, since we always have the equalities $\mathcal{C}\left(R/I; \ttt\right) = \mathcal{C}\left(R/\init_>(I); \ttt\right)$ and $\mathcal{G}\left(R/I; \ttt\right) = \mathcal{G}\left(R/\init_>(I); \ttt\right)$.
	These two equalities follow from the fact that these two notions of multidegree are determined by the Hilbert series (see \autoref{def_multdeg_pol} and \autoref{def_G_multdeg}).

\section{Multidegrees of projections}
	\label{sect_projections}
	
	In this section, we study the multidegrees of all the possible natural projections of an integral multiprojective scheme $X \subset \PP_\kk^{m_1} \times_\kk \cdots \times_\kk \PP_\kk^{m_p}$. 
	We show that the maximal multidegree of a projection is always smaller or equal than the maximal multidegree of $X$.
	Even more, we prove that any multidegree of a projection must divide some multidegree of $X$. 
	The setup below is used throughout this section.

	\begin{setup}
		\label{setup_projections}
		Let $\kk$ be a field and $S = \kk[x_1,\ldots,x_n]$ be a standard $\NN^p$-graded polynomial ring. 
		Suppose that $\multProj(S) = \PP = \PP_\kk^{m_1} \times_\kk \cdots \times_\kk \PP_\kk^{m_p}$.
		Let $X = \multProj(R) \subset \PP$ be an integral closed subscheme and $R$ be an $\NN^p$-graded quotient ring of $S$ that is a domain.
	\end{setup}
	
	For a closed subscheme $Z  \subset  \PP' = \PP_\kk^{r_1} \times_\kk \cdots \times_\kk \PP_\kk^{r_k}$, we consider the following number 
	$$
	\text{MDeg}_{\PP'}(Z) \;:=\; \max\lbrace \deg_{\PP'}^\bn(Z) \mid \bn \in \NN^k \text{ with } |\bn| = \dim(Z)\rbrace. 
	$$
	
	The main result of this section is the theorem below. 
	
	\begin{theorem}
		\label{thm_mdeg_proj}
		Assume \autoref{setup_projections}.
		Let $\fJ = \{j_1,\ldots,j_k\} \subset [p]$ and set $\PP' = \PP_\kk^{m_{j_1}} \times_\kk \cdots \times_\kk \PP_\kk^{m_{j_k}}$.
		Then we have the inequality 
		$$
		{\rm MDeg}_{\PP'}\left(\Pi_\fJ(X)\right) \;\le\; {\rm MDeg}_\PP(X).
		$$
		Moreover, for any $\dd \in \NN^k$ with $|\dd| = \dim\left(\Pi_\fJ(X)\right)$ and such that $\deg_{\PP'}^\dd\left(\Pi_\fJ(X)\right) > 0$, there exists some $\bn \in \NN^p$ with $|\bn| = \dim(X)$ and such that $\deg_\PP^\bn(X) > 0$ and $\deg_{\PP'}^\dd\left(\Pi_\fJ(X)\right)$ divides $\deg_\PP^\bn(X)$.
	\end{theorem}

	Before providing the proof of the theorem we need some preparatory results.
	
	\begin{lemma}
		\label{lem_proj_same_dim}
		Consider the natural projection $\Pi_{[p-1]}:\PP \rightarrow \PP' = \PP_\kk^{m_1} \times_\kk \cdots \PP_\kk^{m_{p-1}}$ and let $Y = \Pi_{[p-1]}(X)$ be the corresponding image.
		Let $\Phi : X \rightarrow Y$ be the restriction of $\Pi_{[p-1]}$ to $X$ and $Y$.
		If $\dim(X) = \dim(Y)$, then
		$$
		\deg_\PP^{\bn,0}(X) \; = \; \deg(\Phi) \deg_{\PP'}^\bn(Y)
		$$
		for all $\bn \in \NN^{p-1}$ with $|\bn| = \dim(X)$;
		here $\deg(\Phi) = \left[K(X) : K(Y)\right]$ denotes the degree of $\Phi$.
	\end{lemma}
	\begin{proof}[First proof.]	
		Let $r = \dim(X)$.
		We write the Hilbert polynomial of $X$ as follows 
		$$
		P_X(t_1,\ldots,t_p) \;=\; \sum_{\substack{\bn \in \NN^p \\ |\bn| \le r}} e(\bn) \binom{t_1+n_1}{n_1}\cdots \binom{t_p+n_p}{n_p}.
		$$
		Choose a positive integer $\omega$ such that $P_X(\nu) = \dim_\kk\left(\left[R\right]_\nu\right)$ for all $\nu \in \NN^p$ satisfying $\nu \ge \omega \cdot \mathbf{1} \in \NN^p$.
		Let $A \subset R$ be the standard $\NN^{p-1}$-graded $\kk$-algebra given by 
		$$
		A \,=\, R_{([p-1])} \,=\, \bigoplus_{\nu_1,\ldots,\nu_{p-1}\ge 0} \left[R\right]_{\nu_1,\ldots,\nu_{p-1},0}.
		$$
		Consider the finitely generated $\NN^{p-1}$-graded $A$-module $M =  \bigoplus_{\nu_1,\ldots,\nu_{p-1}\ge 0} \left[R\right]_{\nu_1,\ldots,\nu_{p-1},\omega}$.
		By construction the Hilbert polynomial of $M$ is given by
		$$
		P_M(\ttt) \,=\, P_X(\ttt,\omega) \,=\, \sum_{\substack{\bn \in \NN^{p-1},\, m \in \NN \\ |\bn| + m \le r}} e(\bn,m) \binom{t_1+n_1}{n_1}\cdots \binom{t_{p-1}+n_{p-1}}{n_{p-1}} \binom{\omega+m}{m}.
		$$
		Regrouping the terms yields the following expression 
		$$
		P_M(\ttt) \,=\, \sum_{\substack{\bn \in \NN^{p-1} \\ |\bn|  \le r}} \left(\sum_{m=0}^{r-|\bn|}e(\bn,m) \binom{\omega+m}{m}\right) \cdot  \binom{t_1+n_1}{n_1}\cdots \binom{t_{p-1}+n_{p-1}}{n_{p-1}}.
		$$
		This forces the equality $e(\bn; M) = e(\bn,0) = \deg_\PP^{\bn,0}(X)$ for all $\bn \in \NN^{p-1}$ with $|\bn|=r$.
		Since $M$ is a torsion-free module over the domain $A$, the associativity formula for mixed multiplicities gives the equalities $\deg_\PP^{\bn,0}(X) = e(\bn; M) = \rank_A(M) e(\bn; A) = \rank_A(M) \deg_{\PP'}^\bn(Y)$ for all $\bn \in \NN^{p-1}$ with $|\bn|=r$.
		As $\dim(X) = \dim(Y)$, the morphism $\Phi$ is generically finite.
		Let $\phi : \Proj(R) \rightarrow \Spec(A)$ be the natural morphism obtained by considering $R$ as an $\NN$-graded ring after setting $[R]_\mu = \bigoplus_{\nu_1,\ldots,\nu_{p-1}\ge 0} [R]_{\nu_1,\ldots,\nu_{p-1},\mu}$ for all $\mu \in \NN$.
		Notice that $\phi$ is a generically finite since $\Phi$ is, and that $\deg(\Phi) = \deg(\phi)$.
		Let $\eta$ the be the generic point of $\Spec(A)$ and $Q = \Quot(A)$ be the field of fractions of $A$.
		Then the generic fiber $Z = \phi^{-1}(\eta) = \Proj(R \otimes_A Q)$ is zero-dimensional (and thus affine, see, e.g., \cite[Proposition 5.20]{GORTZ_WEDHORN}) and we have the equality 
		$$
		\deg(\phi) \,= \, \dim_Q(\phi^{-1}(\eta)) \,=\, \dim_Q\left(\HH^0(Z, \OO_Z(\mu))\right) 
		$$
		for all $\mu \in \ZZ$. 
		In particular, this implies that $\deg(\phi) = \dim_Q\left(\HH^0(Z, \OO_Z(\omega))\right) = \dim_Q\left(\left[R \otimes_A Q\right]_\omega\right) = \rank_A(M)$, after possibly increasing $\omega$ and making it large enough.
		By combining everything, we obtain the claimed equality $
		\deg_\PP^{\bn,0}(X) \; = \; \deg(\Phi) \deg_{\PP'}^\bn(Y)
		$
		for all $\bn \in \NN^{p-1}$ with $|\bn| = r$.
	\end{proof}
\begin{proof}[Second proof.]
	For each $1 \le i \le p-1$, let $L_i = \OO_{\PP'}\left(\ee_i'\right)$ be the line bundle on $\PP' = \PP_\kk^{m_1}\times_\kk\cdots\times_\kk\PP_\kk^{m_{p-1}} $ corresponding with the $i$-th elementary vector $\ee_i' \in \NN^{p-1}$.
	To simplify notation, set $f = \Pi_{[p-1]}$ and let $r = \dim(X)$.
	From the projection formula (see \cite[Proposition 2.5]{FULTON_INTER}) and the fact that $f_*[X] = \deg(\Phi) [Y]$ (see \cite[\S 1.4]{FULTON_INTER}),  we obtain the equalities
	\begin{align*}
		\deg_\PP^{\bn,0}(X) &= \int_{\PP} c_1(f^*(L_1))^{n_1}\cdots c_1(f^*(L_{p-1}))^{n_{p-1}} \cap [X]\\
		&= \deg(\Phi)\int_{\PP'} c_1(L_1)^{n_1}\cdots c_1(L_{p-1})^{n_{p-1}} \cap [Y] \\
		&= \deg(\Phi) \deg_{\PP'}^\bn(Y)
	\end{align*}
	for all $\bn = (n_1,\ldots,n_{p-1}) \in \NN^{p-1}$ with $|\bn| = r$. 
	So, the result follows.
\end{proof}
	
	We shall use a standard type of generic Bertini theorem, much in the spirit of \cite[\S 1.5]{JOINS_INTERS}, that is quite flexible and useful for our purposes.
	To that end, we introduce the following notation.

	\begin{notation}
		\label{nota_generic}
		Let $\{y_0,\ldots,y_\ell\}$ be a basis of the $\kk$-vector space $[S]_{\ee_p}$.
		Consider a purely transcendental field extension $\LL := \kk(z_0,\ldots,z_\ell)$ of $\kk$, and set $S_\LL := S \otimes_\kk \LL$, $R_\LL := R \otimes_\kk \LL$ and $X_\LL := X \otimes_\kk \LL = \multProj\left(R_\LL\right) \subset  \PP_\LL :=   \PP \otimes_\kk \LL = \PP_\LL^{m_1} \times_\LL \cdots \times_\LL \PP_\LL^{m_p}$.
		We say that $y := z_0y_0 + \cdots + z_\ell y_\ell \in {\left[S_\LL\right]}_{\ee_p}$ is the \emph{generic element} of ${\left[S_\LL\right]}_{\ee_p}$.
	\end{notation}

	\begin{remark}
		\label{rem_infinite_field}
		If $\kk(\xi)$ is a purely transcendental field extension over $\kk$,  then $R \otimes_\kk \kk(\xi)$ is also a domain (see, e.g., \cite[Remark 3.9]{POSITIVITY}).
	\end{remark}

	\begin{remark}[{\cite[Corollary 1.5.9, Corollary 1.5.10]{JOINS_INTERS}}]
		\label{rem_gen_Bertini}
		Assume $\dim(X) \ge 1$. Then $X_\LL \cap V(y)$ is also an integral scheme.
	\end{remark}
	\begin{proof}
		We may restrict to an affine open subscheme $U = \Spec(A)$ with $A$ a domain, and assume that the restriction to $A$ of one of the $y_i$'s, say $y_0$, becomes a unit.
		Notice that $U_\LL \cap V(y) = \Spec\left((A \otimes_\kk \LL)/(y)\right)$ and that $(A \otimes_\kk \LL)/(y)$ is a localization of $A[z_0,z_1\ldots,z_\ell]/(z_0 + \frac{y_1}{y_0} z_1 + \cdots + \frac{y_\ell}{y_0} z_\ell)$.
		From the fact that $(z_0 + \frac{y_1}{y_0} z_1 + \cdots + \frac{y_\ell}{y_0} z_\ell)$ is a prime ideal, it follows that $(A \otimes_\kk \LL)/(y)$ is a domain. 
		This concludes the proof.
	\end{proof}

	We recall the following result from \cite{POSITIVITY} that will be essential.
	
	\begin{theorem}
		\label{thm_dim_proj}
		Assume \autoref{setup_projections} and \autoref{nota_generic}, and suppose that $\dim(\Pi_{p}(X)) \ge 1$.
		Then the following statements hold: 
		\begin{enumerate}[\rm (i)]
			\item 
			$
			\dim\left(\Pi_{\fJ}(X_\LL \cap V(y))\right) \;=\; \min\big\{ \dim(\Pi_{\fJ}(X)), \dim(\Pi_{\fJ \cup \{p\}}(X))-1 \big\}
			$
			for all $\fJ \subseteq [p]$.
			\item $\deg_{\PP_\LL}^{\bn - \ee_p}\left(X_\LL \cap V(y)\right) = \deg_{\PP_\LL}^\bn(X_\LL)$ for all $\bn = (n_1,\ldots,n_p) \in \NN^p$ with $|\bn| = \dim(X)$ and $n_p \ge 1$.
		\end{enumerate}
	\end{theorem}	
	\begin{proof}
		We may assume that the field $\kk$ is infinite; indeed, by changing $\kk$ by a purely transcendental field extension we can keep all the assumptions (see \autoref{rem_infinite_field}).
		
		Let $x \in [R]_{\ee_p}$ be a general element.
		From \cite[Theorem 3.7]{POSITIVITY}, we obtain that 
		$$
		\dim\left(\Pi_{\fJ}(X \cap V(x))\right) \;=\; \min\big\{ \dim(\Pi_{\fJ}(X)), \dim(\Pi_{\fJ \cup \{p\}}(X))-1 \big\}
		$$
		for all $\fJ \subset [p]$.
		As a consequence of \cite[Lemma 3.7, Lemma 3.9]{cidruiz2021mixed}, we get that $\deg_{\PP}^{\bn - \ee_p}\left(X \cap V(x)\right) = \deg_{\PP}^\bn(X)$ for all $\bn = (n_1,\ldots,n_p) \in \NN^p$ with $|\bn| = \dim(X)$ and $n_p \ge 1$.
		Therefore, the results of both parts (i) and (ii) follow by applying \cite[Lemma 3.10]{POSITIVITY} and by utilizing the two latter formulas for cutting with a general element $x \in [R]_{\ee_p}$.
	\end{proof}

	We are now ready for the proof of the main result of this section.

	\begin{proof}[Proof of \autoref{thm_mdeg_proj}]
		Notice that it suffices to consider the case $\fJ = [p-1] \subset [p]$.
		Let $Y$ be the image of $X$ under the natural projection $\Pi_{[p-1]} : \PP \rightarrow \PP' = \PP_\kk^{m_1} \times_\kk \cdots \times_\kk \PP_\kk^{m_{p-1}}$.
		We proceed by induction on $s = \dim(X) - \dim(Y)$.
		If $s = 0$, then the result follows directly from \autoref{lem_proj_same_dim}.
		
		Suppose that $s > 0$. 
		This implies that $\dim(\Pi_p(X)) \ge 1$ since the function $r(\fJ) = \dim\left(\Pi_{\fJ}(X)\right)$ is a rank function (see \cite[Proposition 5.1]{POSITIVITY}).
		We substitute $X$ and $Y$ by $X_\LL$ and $Y_\LL$, respectively, and we consider the generic element $y$ in $\left[S_\LL\right]_{\ee_p}$.
		Let $X'= X_\LL \cap V(y) \subset \PP_\LL$ and $Y' = \Pi_{[p-1]}(X') \subset \PP_\LL'$.
		From \autoref{rem_gen_Bertini} and \autoref{thm_dim_proj}, we obtain that $X'$ is an integral scheme and the equalities 
		$$
		\dim\left(X'\right) \,=\, \dim(X)-1 \quad \text{ and } \quad \dim(Y') \,=\, \min\{\dim(X)-1, \dim(Y)\} =  \dim(Y).
		$$
		This implies that $Y'$ and $Y_\LL$ have the same dimension, and thus we necessarily get $Y' = Y_\LL$ because $Y' \subset Y_\LL$ by construction.
		Also, for all $\bn = (n_1,\ldots,n_p) \in \NN^p$ such that $|\bn| = \dim(X)$ and $n_p \ge 1$, we have that $\deg_{\PP_\LL}^{\bn-\ee_p}(X') = \deg_\PP^{\bn}(X)$.
		Therefore,  since $\dim(X') - \dim(Y') = s-1$, we can apply the inductive hypothesis to the scheme $X'$ and the corresponding image $Y'=Y_\LL$ under the projection $\Pi_{[p-1]}$.
		This completes the proof of the theorem.
	\end{proof}

	We illustrate the result of \autoref{thm_mdeg_proj} with the following example that constitutes a  variation of \cite[Example 1.1]{POSITIVITY}.
		
	\begin{example}
		\label{examp_mdeg_project}
		Let $A=\kk[v_1,v_2,v_3][w_1,w_2,w_3]$ be $\NN^3$-graded with $\deg(v_{i})=\mathbf{0} \in \NN^3$ and $\deg(w_{i})=\ee_i \in \NN^3$.
		Let 
		$
		S = \kk\left[x_0,\ldots,x_3\right]\left[y_0,\ldots,y_3\right]\left[z_0,\ldots,z_3\right]
		$ be $\NN^3$-graded with
		 $\deg(x_i)=\ee_1$, $\deg(y_i)=\ee_2$ and $\deg(z_i)=\ee_3$.
		Consider the $\NN^3$-graded $\kk$-algebra homomorphism 
		$$
		\varphi = S \rightarrow A, \qquad 
		\begin{array}{llll}
			x_0 \mapsto  w_1, & x_1 \mapsto v_1^2w_1, & x_2 \mapsto v_1^2w_1, & x_3 \mapsto v_1w_1, \\
			y_0 \mapsto  w_2, & y_1 \mapsto v_1w_2, & y_2 \mapsto v_2w_2, & y_3 \mapsto (v_1^2+v_2^2)w_2,  \\
			z_0 \mapsto  w_3, & z_1 \mapsto v_1^2w_3, & z_2 \mapsto v_2w_3, & z_3 \mapsto v_3^2w_3. 
		\end{array}
		$$
		Let $P \subset S$ be the $\NN^3$-graded prime ideal $P = \Ker(\varphi)$, which is given by
		$$
		P \,= \, \left(
		\begin{array}{c}
				{x}_{1}-{x}_{2}, \, {y}_{3}{z}_{0}-{y}_{0}{z}_{1}-{y}_{2}{z}_{2}, \, {y}_{2}{z}_{0}-{y}_{0}{z}_{2},\, {x}_{2}{z}_{0}-{x}_{0}{z}_{1},\, {y}_{1}^{2}+{y}_{2}^{2}-{y}_{0}{y}_{3}, \, {x}_{3}{y}_{0}-{x}_{0}{y}_{1}\\
				{x}_{2}{y}_{0}-{x}_{3}{y}_{1}, \, {x}_{0}{x}_{2}-{x}_{3
				}^{2}, \, {y}_{0}{y}_{2}{z}_{1}+{y}_{2}^{2}{z}_{2}-{y}_{0}{y}_{3}{z}_{2}, \, {x}_{3}{y}_{2}{z}_{1}-{x}_{2}{y}_{1}{z}_{2}, \,{x}_{0}{y}_{2}{z}_{1}-{x}_{3}{y}_{1}{z}_{2},\\
				{x}_{3}{y}_{1}{z}_{1}-{x}_{0}{y}_{3}{z}_{1}+{x}_{2}{y}_{2}{z}_{2}, \, {x}_{3}{y}_{1}{z}_{0}-{x}_{0}{y}_{0}{z}_{1}, \, {x}_{3}^{2}{z}_{0}-{x}_{0}^{2}{z}_{1}
		\end{array}
		\right).
		$$
		Let $X = \multProj(S/P) \subset  \PP_\kk^3\times_\kk \PP_\kk^3 \times_\kk \PP_\kk^3$ be the corresponding integral closed subscheme.
		The multidegree polynomial of $X$ is equal to 
		$$
		\mathcal{C}(X; t_1,t_2,t_3) \,=\, 2\,{t}_{1}^{3}{t}_{2}^{3}+4\,{t}_{1}^{3}{t}_{2}^{2}{t}_{3}+2\,{t}_{1}^{3}{t}_{2}{t}_{3}^{2}+
		2\,{t}_{1}^{2}{t}_{2}^{3}{t}_{3}+4\,{t}_{1}^{2}{t}_{2}^{2}{t}_{3}^{2}.
		$$
		For ease of exposition, we use the variables indexed by $\fJ$ to express the multidegree polynomial of a projection $\Pi_{\fJ}(X)$.
		Under this notation, we obtain that the multidegree polynomials of the possible projections are described in the following table:
		\begin{equation*}		
		\setlength{\arraycolsep}{15pt}
		\begin{array}{lll}
				\Pi_1(X) \subset  \PP_\kk^3, &   \dim(\Pi_1(X)) = 1,  & \mathcal{C}(\Pi_1(X); t_1) = 2\,{t}_{1}^{2}  \\
				\Pi_2(X) \subset  \PP_\kk^3, &  \dim(\Pi_2(X)) = 2, & \mathcal{C}(\Pi_2(X); t_2) = 2\,{t}_{2} \\ 
				\Pi_3(X) \subset  \PP_\kk^3, &  \dim(\Pi_3(X)) = 3, & \mathcal{C}(\Pi_3(X); t_3) = 1 \\
				\Pi_{1,2}(X) \subset  \PP_\kk^3 \times_\kk \PP_\kk^3, &  \dim(\Pi_{1,2}(X)) = 2, & \mathcal{C}(\Pi_{1,2}(X); t_1,t_2) = 2\,{t}_{1}^{3}{t}_{2}+4\,{t}_{1}^{2}{t}_{2}^{2} \\
				\Pi_{1,3}(X) \subset  \PP_\kk^3 \times_\kk \PP_\kk^3, &  \dim(\Pi_{1,3}(X)) = 3, & \mathcal{C}(\Pi_{1,3}(X); t_1,t_3) = 2\,{t}_{1}^{3}+2\,{t}_{1}^{2}{t}_{3} \\
				\Pi_{2,3}(X) \subset  \PP_\kk^3 \times_\kk \PP_\kk^3, &  \dim(\Pi_{2,3}(X)) = 3, & \mathcal{C}(\Pi_{2,3}(X); t_2,t_3) = 2\,{t}_{2}^{3}+4\,{t}_{2}^{2}{t}_{3}+2\,{t}_{2}{t}_{3}^{2}. \\
		\end{array}
	\end{equation*}	
	Here one may see that the multidegrees of the projections behave as predicted by \autoref{thm_mdeg_proj}.
	The computations of this example can be carried out by hand or by utilizing \texttt{Macaulay2} \cite{M2}.
	\end{example}

	The next simple example shows that, if we drop the condition of \autoref{thm_mdeg_proj} that $R$ is a domain, then the multidegrees  can grow arbitrarily under projections.

	\begin{example}
		\label{examp_grow_mdeg_proj}
		Suppose $S = \kk[x_0,x_1,x_2][y_0,y_1,y_2]$ is $\NN^2$-graded with $\deg(x_i) = \ee_1$ and $\deg(y_i) = \ee_2$.
		Let $a >0$ be a positive integer.
		Consider the monomial ideal $J = \left(x_0^2,x_0x_1,x_1y_0,y_0^a\right) \subset S$ and its corresponding closed subscheme $Z = \multProj(S/J) \subset  \PP_\kk^2 \times_\kk \PP_\kk^2$.
		One can check that 
		$$
		\mathcal{C}(Z;t_1,t_1) = t_1t_2 \quad \text{ and } \quad \deg_{\PP_\kk^2}(\Pi_2(Z)) = a.
		$$
		This follows at once from the fact $J = (x_0,y_0) \cap (x_0^2,x_1,y_0^a)$ is a primary decomposition of $J$.
	\end{example}
	
	As a direct consequence of \autoref{thm_mdeg_proj}, we obtain that the projection of a multiplicity-free variety is also multiplicity-free. 
		A closed subscheme $Z  \subset  \PP' = \PP_\kk^{r_1} \times_\kk \cdots \times_\kk \PP_\kk^{r_k}$ is said to be \emph{multiplicity-free} if $\deg_{\PP'}^\bn(Z)$ is either $0$ or $1$ for all $\bn \in \NN^k$ with $|\bn| = \dim(Z)$ (that is, ${\rm MDeg}_{\PP'}(Z)=1$).

	\begin{corollary}	
		\label{cor_mult_free}
		If $X \subset   \PP$ is multiplicity-free, then $\Pi_\fJ(X)$ is multiplicity-free for all $\fJ \subset [p]$.
	\end{corollary}

	\section{The multidegrees of prime ideals}
	\label{sect_multdeg_primes}

	In this section, we investigate several special properties that are enjoyed by the multidegrees of prime ideals.
	We use the following setup throughout this section. 
	
	\begin{setup}
		\label{setup_prime_ideals}
		Let $\kk$ be a field and $S = \kk[x_1,\ldots,x_n]$ be a standard $\NN^p$-graded polynomial ring.
	\end{setup}

	The theorem below proves that the multidegree polynomial and the $\mathcal{G}$-multidegree polynomial coincide for a prime ideal.

	\begin{theorem}
		Assume \autoref{setup_prime_ideals}.
		Let $P \subset S$ be an $S$-homogeneous prime ideal.
		Then we have the equality 
		$$
		\mathcal{C}(S/P;\ttt) \,= \, \mathcal{G}(S/P; \ttt).
		$$
	\end{theorem}
	\begin{proof}
		Choose a monomial order $>$ on $S$.
		By using \cite{KS_INIT_PRIMES}, we obtain that $\sqrt{\init(P)}$ is equidimensional and connected in codimension one (see also \cite[Appendix~1]{HUNEKE_LOC_COHOM}).
		For any monomial prime ideal $L = (x_{j_1},\ldots,x_{j_k}) \subset S$, we have 
		$$
		\mathcal{C}(S/L;\ttt) \, = \, \ttt^\mathbf{a}  \quad \text{ and } \quad  \mathcal{G}(S/L;\ttt)  \, = \,   \ttt^\mathbf{a}
		$$
		where $\mathbf{a} = (a_1,\ldots,a_p) \in \NN^p$ with $a_i = |\lbrace l \mid \deg(x_{j_l}) = \ee_i \rbrace|$ (see \autoref{rem_multdeg_facts}).
		
		Let $P_1,\ldots,P_h, P_{h+1}, \ldots, P_{h+s}$ be the associated primes of $\init(P)$, where $P_1,\ldots,P_h$ are the minimal ones. 
		We choose a prime filtration 
		$$
		0 = M_0 \subset M_1 \subset \cdots \subset M_r = S/\init(P)
		$$
		where $M_j/M_{j-1} \cong S/L_j$ and $L_j \subset S$ is a monomial prime ideal (see, e.g., \cite[Proposition 3.7]{EISEN_COMM}).
		All the associated primes of $\init(P)$ must appear among the $L_j$'s and all the $L_j$'s must contain a minimal prime of $\init(P)$.
		Therefore, by utilizing the additivity of Hilbert functions over this prime filtration and the above remarks, we obtain 
		$$
		\mathcal{C}\left(S/\init(P);\ttt\right) \, = \, \mathcal{G}\left(S/\init(P);\ttt\right) \, = \, \sum_{j=1}^h \text{length}_{S_{P_j}}\left((S/\init(P))_{P_j}\right) \mathcal{C}(S/P_j;\ttt);
		$$
		indeed: if $L_{k_1} \subset L_{k_2}$ then $\mathcal{C}(S/L_{k_1}; \ttt)$ divides $\mathcal{C}(S/L_{k_2}; \ttt)$; all the minimal $P_j$'s have the same codimension; and by localizing we get that the number of times that a minimal $P_j$ appears among the $L_k$'s is equal to the length of the Artinian local ring $(S/\init(P))_{P_j}$.
		This concludes the proof because $S/P$ and $S/\init(P)$ have the same multigraded Hilbert function.
	\end{proof}

	The remainder of this section is dedicated to multigraded generic initial ideals.
	Thus we further specify the notation by fixing the following setup. 
	
	\begin{setup}
		\label{setup_gin}
		Let $\kk$ be an infinite field, $\bm = (m_1,\ldots,m_p) \in \ZZ_+^p$ be a vector of positive integers, and 
		$$
		S = \kk\left[x_{i,j} \mid 1 \le i \le p, 0 \le j \le m_i \right]
		$$ 
		be a standard $\NN^p$-graded polynomial with $\deg(x_{i,j}) = \ee_i \in \NN^p$. 
		Hence we have $\multProj(S) = \PP = \PP_\kk^{m_1} \times_\kk \cdots \times_\kk \PP_\kk^{m_p}$.
		The group naturally acting on $S$ as the group of multigraded $\kk$-algebra isomorphisms is $G = \text{GL}_{m_1+1}(\kk) \times \cdots \times \text{GL}_{m_p+1}(\kk)$, where $\text{GL}_{m_i+1}(\kk)$ is the group of invertible $(m_i+1)\times (m_i+1)$ matrices over $\kk$.
		The \emph{Borel subgroup of $G$} is given by $B = B_{m_1+1}(\kk) \times \cdots \times B_{m_p+1}(\kk)$, where $B_{m_i+1}(\kk)$ is the subgroup of $\text{GL}_{m_i+1}(\kk)$ consisting of upper triangular invertible matrices.
		Let $>$ be a monomial order on $S$ that satisfies $x_{i,0} > x_{i,1} > \cdots > x_{i,m_i}$ for all $1 \le i \le p$. 
	\end{setup}

	Much as in the singly-graded case (see \cite[\S 15.9]{EISEN_COMM}), we have the following definition. 
	
	\begin{definition}
		Let $I \subset S$ be an $S$-homogeneous ideal.
		The \emph{multigraded generic initial ideal} $\gin_>(I)$ of $I$ with respect to $>$ is the ideal $\init_>(g(I))$, where $g$ belongs to a Zariski dense open subset $\mathcal{U} \subset G$.
	\end{definition}
	
	Whenever the used monomial order $>$ is clear from the context, we shall write $\init(I)$ and $\gin(I)$ instead of $\init_>(I)$ and $\gin_>(I)$, respectively.
	An $S$-homogeneous ideal $I \subset S$ is said to be \emph{Borel-fixed} if $g(I) = I$ for all $g \in B$.
	Similarly to the singly-graded setting, it can be shown that $\gin(I)$ is Borel-fixed (see \cite[Theorem 15.20]{EISEN_COMM}). 	
	
	\begin{remark}
		One can extend the known properties of Borel-fixed ideals in the singly-graded setting (see \cite[Chapter 15]{EISEN_COMM}) to obtain the following statements:
		\begin{enumerate}[\rm (i)]
			\item Any Borel-fixed prime ideal in $S$ is of the form
			$$
			P_\mathbf{a} \,=\, \left(x_{i,j} \mid 1 \le i \le p \text{ and } 0 \le j < a_i\right)
			$$
			for some $\mathbf{a} = (a_1,\ldots,a_p) \in \NN^p$ (see \cite[Lemma 3.1]{CDNG_MINORS}).
			\item An $S$-homogeneous Borel-fixed ideal $I \subset S$ is a monomial ideal and all its associated primes are also Borel-fixed (see \cite[Lemma 3.2]{CDNG_MINORS}).
		\end{enumerate}
	\end{remark}

	The following theorem shows that the radical of the multigraded generic initial ideal of any prime is Cohen-Macaulay.
	This demonstrates a very special behavior of prime ideals in a multigraded setting.
	
	\begin{theorem}
		\label{thm_sqrt_gin_CM}
		Assume \autoref{setup_gin}.
		Let $P \subset S$ be an $S$-homogeneous prime ideal.
		Then $\sqrt{\gin(P)}$ is a Cohen-Macaulay ideal.
	\end{theorem}
\begin{proof}
		Let $J = \sqrt{\gin(P)} \subset S$.
		Let $P_1,\ldots,P_h$ be the minimal primes of $J$.
		As a consequence of \cite{KS_INIT_PRIMES}, all the $P_k$'s have the same codimension.
		Each $P_k$ is Borel-fixed, and so we can write
		$$
		P_k  \, = \, \left(x_{i,j} \mid 1 \le i \le p, 0 \le j < a_{k,i}\right)
		$$
		for some $\mathbf{a}_k = (a_{k,1},\ldots,a_{k,p}) \in \NN^p$ with $|\mathbf{a}_k| = \codim(P)$.
		With this notation we now have
		$$
		\mathcal{C}(S/P; \ttt)  \,= \, \mathcal{C}(S/\gin(P); \ttt)  \,= \, \sum_{k=1}^{h} \text{length}_{S_{P_k}}\left((S/\gin(P))_{P_k}\right) \, \ttt^{\mathbf{a}_k} 
		\quad 
		\text{ and }
		\quad 
		\mathcal{C}(S/J;  \ttt) \,= \, \sum_{k=1}^{h}  \, \ttt^{\mathbf{a}_k}. 
		$$
		The Alexander dual of $J$ is a monomial ideal $K \subset S$ given by
		$
		K  =  \left( \xx_{\mathbf{a}_1}, \ldots, \xx_{\mathbf{a}_h} \right) 
		$
		where
		$$
		\xx_{\mathbf{a}_k} = \prod_{1 \le i \le p, \,0 \le j < a_{k,i}} x_{i,j};
		$$
		see \cite[Corollary 1.5.5]{HERZOG_HIBI}.
		The Eagon-Reiner theorem (see, e.g., \cite[Theorem 8.1.9]{HERZOG_HIBI}) yields that $J$ is Cohen-Macaulay if and only if $K$ has a linear resolution.
		Thus we concentrate on proving that $K$ has a linear resolution.

		From \autoref{thm_pos_multdeg}, we obtain that the support of $\mathcal{C}(S/P; \ttt)$, which coincides with the set of lattice points $\lbrace \mathbf{a}_1,\ldots,\mathbf{a}_h \rbrace \subset \NN^p$, is a discrete polymatroid.
		We construct the following monomial ideal 
		$$
		M = \left( \xx_{\mathbf{a}_1}',\ldots, \xx_{\mathbf{a}_h}' \right) \quad \text{ with } \quad \xx_{\mathbf{a}_k}' = x_{1,0}^{a_{k,1}} x_{2,0}^{a_{k,2}} \cdots x_{p,0}^{a_{k,p}}
		$$
		in the polynomial subring $\kk[x_{1,0}, x_{2,0}, \ldots, x_{p,0}] \subset S$.
		By construction and following the notation of \cite[\S 12.6]{HERZOG_HIBI}, we say that that $M$ is a polymatroidal ideal, and so the ideal $M$ has a linear resolution by \cite[Theorem 12.6.2, Proposition 8.2.1]{HERZOG_HIBI}.
		Notice that $K$ can be seen naturally as the polarization of $M$ by sending $\xx_{\mathbf{a}_k}'$ to $\xx_{\mathbf{a}_k}$.
		Finally, by standard properties of polarization (see \cite[\S 1.6]{HERZOG_HIBI}), it follows that $K$ also has a linear resolution. 
		This concludes the proof of the theorem.
\end{proof}

\begin{remark}
	For the more combinatorially inclined reader it should be the mentioned that the same proof of \autoref{thm_sqrt_gin_CM} shows that the simplicial complex associated to the radical monomial ideal $\sqrt{\gin(P)}$ is shellable (see \cite[Proposition 8.2.5]{HERZOG_HIBI}).
\end{remark}

\begin{remark}
	\label{rem_not_CM_init}
	One may find many examples of a prime ideal $P$ and a monomial order $>$ such that $\sqrt{\init_>(P)}$ is not Cohen-Macaulay.
	For a simple instance, see \cite[Remark 2.8(1)]{CV_SQRFREE}. 
	Also, for more examples of this type and related results, see the book \cite[Chapter 5]{BOOK_UPCOMMING}.
\end{remark}

For an ideal $I \subset S$, we consider the following invariant 
$$
\text{MLength}(I) \;:=\; \max \big\lbrace \text{length}_{S_\pp}\left((S/I)_\pp\right) \mid \pp \in \Min_S(S/I) \big\rbrace
$$
that measures the maximal length of the minimal primary components of $I$.
The next theorem shows that, for a given $S$-homogeneous prime ideal, the multiplicities of the minimal primary components of the multigraded generic initial ideal have a very restrictive behavior with respect to the natural projections.

\begin{theorem}
	\label{thm_proj_gin}
	Assume \autoref{setup_gin}.
	Let $P \subset S$ be an $S$-homogeneous prime ideal.
	Let $\fJ = \{j_1,\ldots,j_k\} \subset [p]$ be a subset.
	Then we have the inequality 
	$$
	{\rm MLength}\left(\gin\left(P_{(\fJ)}\right)\right) \;\le\; {\rm MLength}\left(\gin\left(P\right)\right).
	$$
	Moreover, for a given minimal prime $\pp \in \Min_{S_{(\fJ)}}\left(S_{(\fJ)}/\gin(P_{(\fJ)})\right)$ of $\gin(P_{(\fJ)})$, we can find a minimal prime $\qqq \in \Min_S\left(S/\gin(P)\right)$ of $\gin(P)$ such that the length of the $\pp$-primary component of $\gin\left(P_{(\fJ)}\right)$ divides the length of the $\qqq$-primary component of $\gin(P)$.
\end{theorem}
\begin{proof}
	As in \autoref{rem_relev_prim}, we can add new variables $x_{i,m_i +1}$ to $S$ with $\deg(x_{i,m_i+1}) = \ee_i$, and then assume that $P$ is relevant. 
	By further specifying $x_{i,m_i} > x_{i,m_i+1}$, we may see that $\gin(P)$ will not involve the new variables $x_{i,m_i+1}$. 
	So, without any loss of generality, we assume that $P$ is a relevant prime ideal.

	Let $I = \gin(P) \subset S$, and
	 $P_1,\ldots,P_h$ be the minimal primes of $I$.
	 For each $1 \le k \le h$,  we know that $
	P_k  =  \left(x_{i,j} \mid 1 \le i \le p, 0 \le j < a_{k,i}\right)
	$
	for some $\mathbf{a}_k = (a_{k,1},\ldots,a_{k,p}) \in \NN^p$ with $|\mathbf{a}_k| = \codim(P)$.
	Hence by utilizing the associativity formula for multidegrees (\autoref{rem_multdeg_facts}(iv)) and \autoref{rem_multdeg_facts}(ii), we obtain
	\begin{equation}
		\label{eq_muldeg_as_multiplicities}
		\mathcal{C}(S/P;\ttt) \,=\,  \mathcal{C}(S/I;\ttt) \,=\, \sum_{k=1}^h \text{length}_{S_{P_k}}\left((S/I)_{P_k}\right) \, \ttt^{\mathbf{a}_k}. 
	\end{equation}
	Let $X = \multProj(S/P) \subset \PP$ be the integral closed subscheme corresponding with $P$.
	Set $r = \dim(X) = m_1 + \cdots + m_p - \codim(P)$.
	From \autoref{thm_two_mdegs} it follows that 
	$$
	\mathcal{C}(S/P;\ttt) \,=\, \sum_{\substack{\bn \in \NN^p \\ |\bn| = r}} \deg_{\PP}^\bn(X) \ttt^{\bm-\bn}.
	$$
	Then comparing coefficients yields the equality 
	$$
	\text{length}_{S_{P_k}}\left((S/I)_{P_k}\right) \,=\, \deg_{\PP}^{\bm - \mathbf{a}_k}(X) 
	$$
	for all $1 \le k \le h$.
	
	Set $\PP' = \PP_\kk^{m_{j_1}} \times_\kk \cdots \times_\kk \PP_\kk^{m_{j_k}}$.
	By applying the same argument to the projection corresponding to $\fJ$, we obtain that for any minimal prime $\pp\subset S_{(\fJ)}$ of $\gin(P_{(\fJ)})$, there exists some $\dd \in \NN^k$ with $|\dd| = \dim\left(\Pi_\fJ(X)\right)$ such that
	$$
		\text{length}_{{\left(S_{(\fJ)}\right)}_{\pp}}\left(\left(S_{(\fJ)}/\gin(P_{(\fJ)})\right)_{\pp}\right) \,=\, \deg_{\PP'}^{\dd}\left(\Pi_\fJ(X)\right).
	$$
	Therefore, the result  follows as a consequence of \autoref{thm_mdeg_proj}.
\end{proof}

We single out the following result that was obtained in \autoref{eq_muldeg_as_multiplicities}.

\begin{remark}
		\label{rem_mdeg_mult_min_primes}
		Let $P \subset S$ be an $S$-homogeneous prime ideal. 
		Let $I = \gin(P) \subset S$. 
		Let $P_1,\ldots,P_h$ be the minimal primes of $I$ and $Q_1,\ldots,Q_h$ be the corresponding minimal primary components. 
		The non-zero coefficients of $\mathcal{C}(S/P;\ttt)$ are in one-to-one correspondence with the multiplicities $\text{length}_{{S_{P_k}}}\left((S/Q_k)_{P_k}\right)$.
\end{remark}

The next example shows that the minimal primary components of $\gin(P_{(\fJ)})$ may be determined by embedded primary components of $\gin(P)$.
However, \autoref{thm_proj_gin} still gives a tight relation between the minimal primary components of $\gin(P_{(\fJ)})$ and  $\gin(P)$.

\begin{example}
	\label{examp_project_gin}
	Here we continue using the prime ideal $P \subset S$ from \autoref{examp_mdeg_project}, where $S$ is the $\NN^3$-graded polynomial ring $S = \kk\left[x_0,\ldots,x_3\right]\left[y_0,\ldots,y_3\right]\left[z_0,\ldots,z_3\right]$ with
	$\deg(x_i)=\ee_1$, $\deg(y_i)=\ee_2$ and $\deg(z_i)=\ee_3$.
	The corresponding multidegree polynomial is given by 
	$$
	\mathcal{C}(R/P; t_1,t_2,t_3) \,=\, 2\,{t}_{1}^{3}{t}_{2}^{3}+4\,{t}_{1}^{3}{t}_{2}^{2}{t}_{3}+2\,{t}_{1}^{3}{t}_{2}{t}_{3}^{2}+
	2\,{t}_{1}^{2}{t}_{2}^{3}{t}_{3}+4\,{t}_{1}^{2}{t}_{2}^{2}{t}_{3}^{2}.
	$$
	As we have pointed out before in this section, since $\gin(P)$ is Borel-fixed and $\sqrt{\gin(P)}$ is equidimensional (we proved that it is even Cohen-Macaulay in \autoref{thm_sqrt_gin_CM}), it follows that the multidegree polynomial determines the minimal primes of $\gin(P)$.
	More precisely, we have the following correspondence between terms of $\mathcal{C}(R/P;\ttt)$ and minimal primary components of $\gin(P)$:
	\begin{equation*}		
		\setlength{\arraycolsep}{10pt}
		\begin{array}{lll}
			2\,t_1^3t_2^3 &   \longleftrightarrow  & \left({x}_{0}, \, {x}_{1}, \, {x}_{2}^{2}, \, {y}_{0}, \, {y}_{1}, \, {y}_{2}\right) \,=:\, M_1 \\
			4\,t_1^3t_2^2t_3 &  \longleftrightarrow  & \left({x}_{0},\, {x}_{1}^{2}, \, {x}_{1}{x}_{2}, \, {x}_{2}^{3}, \, {y}_{0}, \, {y}_{1}, \, {z}_{0}\right) \,=:\, M_2 \\ 
			2\,t_1^3t_2t_3^2 &  \longleftrightarrow  & \left({x}_{0},\, {x}_{1}, \, {x}_{2}^{2}, \, {y}_{0}, \, {z}_{0}, \, {z}_{1}\right) \,=:\, M_3\\ 
			2\,t_1^2t_2^3t_3 &  \longleftrightarrow  & \left({x}_{0}, \, {x}_{1}^{2}, \, {y}_{0}, \, {y}_{1}, \, {y}_{2}, \, {z}_{0}\right) \,=:\, M_4\\ 
			4\,t_1^2t_2^2t_3^2 &  \longleftrightarrow  & \left({x}_{0}, \, {x}_{1}^{2}, \, {y}_{0}, \, {y}_{1}^{2}, \, {z}_{0}, \, {z}_{1}\right) \,=:\, M_5.\\
		\end{array}
	\end{equation*}	
	One can compute the whole expression of $\gin(P)$ by utilizing \texttt{Macaulay2} \cite{M2}, and one may check that $\gin(P)$ has a total of $9$ primary components and that all the embedded ones have codimension $7$.
	We consider the projection $\fJ = \{2,3\}$.
 	The corresponding projected ideal is given by 
 	$$
 	Q = P_{({2,3})} = \left({y}_{3}{z}_{0}-{y}_{0}{z}_{1}-{y}_{2}{z}_{2},\, {y}_{2}{z}_{0}-{y}_{0}{z}_{2},\, {y}_{1}^{2}+{y}_{2}^{2}-{y}_{0}{y}_{3}\right) \,\subset  T = S_{(2,3)}.
 	$$
 	We already saw that the multidegree polynomial of this projection is equal to 
 	$$
 	\mathcal{C}(T/Q; t_2,t_3) = 2\,{t}_{2}^{3}+4\,{t}_{2}^{2}{t}_{3}+2\,{t}_{2}{t}_{3}^{2}.
 	$$
 	The correspondence between terms of $\mathcal{C}(T/Q; t_2,t_3)$ and minimal primary components of $\gin(Q)$ is depicted below:
 		\begin{equation*}		
 		\setlength{\arraycolsep}{10pt}
 		\begin{array}{lll}
 			2\,t_2^3 &   \longleftrightarrow  & \left({y}_{0}, \, {y}_{1}, \, {y}_{2}^{2}\right)  \\
 			4\,t_2^2t_3 &  \longleftrightarrow  & \left({y}_{0}^{2}, \, {y}_{0}{y}_{1}, \, {y}_{1}^{3}, \, {z}_{0}\right) \\ 
 			2\,t_2t_3^2 &  \longleftrightarrow  & \left({y}_{0}^{2}, \, {z}_{0}, \, {z}_{1}\right).\\ 
 		\end{array}
 	\end{equation*}	
 	Again, $\gin(Q)$ can be computed by using \texttt{Macaulay2} \cite{M2}, or just by realizing that $\gin(Q) = \gin(P)_{(2,3)}$.
 	It turns out that $\gin(Q)$ has only three primary components and that all of them are minimal.
 	Notice that $M_1 \cap T = (y_0,y_1,y_2)$, $M_2 \cap T = (y_0,y_1,z_0)$, $M_3 \cap T = (y_0,z_0,z_1)$, $M_4 \cap T = (y_0,y_1,y_2,z_0)$ and $M_5 \cap T = (y_0, y_1^2, z_0, z_1)$.
 	So, it follows that all the primary components of $\gin(Q)$ come from embedded components of $\gin(P)$, and yet the length of any (minimal) component of $\gin(Q)$ divides the length of some minimal component of $\gin(P)$; as predicted by \autoref{thm_proj_gin}. 
\end{example}

\begin{remark}
	\label{rem_counter_gin_not_prime}
	If we drop the prime condition, then the result of \autoref{thm_proj_gin} may not hold.
	Indeed, let $J = \left(x_0^2,x_0x_1,x_1y_0,y_0^a\right) \subset S$ be the $\NN^2$-graded ideal of \autoref{examp_grow_mdeg_proj}, then we have that $\text{MLength}\left(\gin(J)\right) = 1$ and $\text{MLength}\left(\gin(J_{(2)})\right) = a$.
\end{remark}

\section{Multiplicity-free varieties}
\label{sect_mult_free}

The main goal of this section is to obtain an alternative proof of the result of Brion \cite{BRION_MULT_FREE} regarding multiplicity-free varieties.
For organizational purposes we divide the section into two subsections.

\subsection{Cohomology and associated primes under flat degenerations}	
	\label{subsect_cohom_assP_degen}
	This subsection contains some technical results that will be needed in our treatment of multiplicity-free varieties.
	Here we study the behavior of cohomology and associated primes under flat degenerations. 
	Our approach is inspired by ideas related to the fiber-full scheme (\cite{LOC_FIB_FULL_SCHEME, FIBER_FULL,FIB_FULL_SCHEME}).
	
\begin{proposition}[{cf.~\cite[Theorem A]{LOC_FIB_FULL_SCHEME}}]
	\label{prop_criterion_flat_cohom}
	Let $(A, \nnn, \kk)$ be a Noetherian local ring with residue field $\kk$.
	Let $\mathbb{Y} \subset \PP_A^r$ be a closed subscheme and $\mathfrak{F}$ be a coherent sheaf on $\mathbb{Y}$.
	Suppose that $\mathfrak{F}$ is flat over $A$.
	Let $Y = \mathbb{Y} \times_{\Spec(A)} \Spec(\kk) \subset \PP_\kk^r$ be the closed subscheme given as the special fiber and $F = \mathfrak{F} \otimes_A \kk$ be the corresponding coherent sheaf on  $Y$.
	
	If \,$\Hom_\kk\left(\HH^i(Y, F),\, \HH^{i+1}(Y, F)\right) = 0$ for all $i \ge 0$, then the following statements hold: 
	\begin{enumerate}[\rm (i)]
		\item $\HH^i(\mathbb{Y}, \mathfrak{F})$ is $A$-flat for all $i \ge 0$.
		\item The natural map $\HH^i(\mathbb{Y}, \mathfrak{F}) \otimes_A B \rightarrow \HH^i(\mathbb{Y} \times_{\Spec(A)} \Spec(B), \mathfrak{F} \otimes_A B)$ is an isomorphism for all $i \ge 0$ and any $A$-algebra $B$.
	\end{enumerate}
\end{proposition}	
\begin{proof}
	For any $A$-algebra $B$, set $\mathbb{Y}_B := \mathbb{Y} \times_{\Spec(A)} \Spec(B)$ and $\mathfrak{F}_B := \mathfrak{F} \otimes_A B$.
	
	First, we prove that for any Artinian quotient ring $C$ of $A$ both statements of the proposition hold.
	We proceed by induction on $\text{length}(C)$.
	The base case $\text{length}(C)=1$ is clear because we necessarily have $C = \kk$.
	Thus, we assume that $\text{length}(C) > 1$.
	We may choose an ideal $\aaa \subset C$ generated by a socle element of $C$ to obtain a short exact sequence 
	$$
	0 \rightarrow \aaa \rightarrow C \rightarrow C' \rightarrow 0
	$$	
	with $\aaa \cong \kk$ and $\text{length}(C') = \text{length}(C)-1$.
	By tensoring with $- \otimes_A \mathfrak{F}$, since $\mathfrak{F}$ is flat over $A$, we obtain the short exact sequence 
	$$
	0 \rightarrow F \rightarrow \mathfrak{F}_C \rightarrow \mathfrak{F}_{C'} \rightarrow 0.
	$$
	Consequently, we get an induced long exact sequence in cohomology
	$$
	\HH^{i-1}(\mathbb{Y}_{C'},\mathfrak{F}_{C'}) \,\xrightarrow{\delta_{i-1}}\, \HH^i(Y, F) \,\rightarrow\, \HH^i(\mathbb{Y}_C,\mathfrak{F}_C)\,\rightarrow\, \HH^i(\mathbb{Y}_{C'},\mathfrak{F}_{C'}) \,\xrightarrow{\delta_i}\, \HH^{i+1}(Y, F).
	$$
	The inductive hypothesis yields the isomorphism $\HH^i(\mathbb{Y}_{C'},\mathfrak{F}_{C'}) \otimes_{C'} \kk \xrightarrow{\cong} \HH^i(Y, F)$.
	Then the map $\delta_i$ gets the following factorization 
	$$
	\HH^i(\mathbb{Y}_{C'},\mathfrak{F}_{C'}) \; \twoheadrightarrow \; \HH^i(\mathbb{Y}_{C'},\mathfrak{F}_{C'}) \otimes_{C'} \kk \cong \HH^i(Y, F) \; \xrightarrow{\beta_i} \; \HH^{i+1}(Y, F)
	$$
	where $\beta_i$ is a $\kk$-linear map.
	The condition $\Hom_\kk\left(\HH^i(Y, F),\, \HH^{i+1}(Y, F)\right) = 0$ implies that $\beta_i = 0$, hence $\delta_i = 0$.
	For all $i \ge 0$, we have proved that the natural map $\HH^i(\mathbb{Y}_C,\mathfrak{F}_C)\rightarrow \HH^i(\mathbb{Y}_{C'},\mathfrak{F}_{C'})$ is surjective, and the natural map $\HH^i(\mathbb{Y}_{C'},\mathfrak{F}_{C'}) \rightarrow \HH^i(Y, F)$ is surjective by induction. 
	Therefore, it follows that the natural map 
	$$
	\HH^i(\mathbb{Y}_{C},\mathfrak{F}_{C}) \,\rightarrow\, \HH^i(Y, F)
	$$
	is also surjective for all $i \ge 0$.
	Finally, by applying \cite[Theorem 12.11]{HARTSHORNE}, we obtain both statements for $C$.
	This establishes the proposition for any Artinian quotient ring of $A$. 
	
	For each $q \ge 1$, let $\mathbb{Y}_q = \mathbb{Y} \times_{\Spec(A)} \Spec(A/\nnn^q)$ and $\mathfrak{F}_q = \mathfrak{F} \otimes_{A} A/\nnn^q$.
	By the above step we have that $\HH^i(\mathbb{Y}_q, \mathfrak{F}_q)$ is flat over $A/\nnn^q$.
	The Theorem on Formal Functions (see \cite[Theorem III.11.1]{HARTSHORNE}, \cite[\href{https://stacks.math.columbia.edu/tag/02OC}{Tag 02OC}]{stacks-project}) yields the following isomorphism 
	$$
	\HH^i(\mathbb{Y}, \mathfrak{F}) \otimes_A \widehat{A} \;\; \xrightarrow{\; \cong \; } \;\; \underset{\leftarrow}{\lim}\;\HH^i(\mathbb{Y}_q, \mathfrak{F}_q).
	$$
	So by applying \cite[\href{https://stacks.math.columbia.edu/tag/0912}{Tag 0912}]{stacks-project} to the inverse system $\left(\HH^i(\mathbb{Y}_q, \mathfrak{F}_q)\right)_{q\ge 1}$, we obtain that $\HH^i(\mathbb{Y}, \mathfrak{F}) \otimes_A \widehat{A}$ is a flat $A$-module.
	Hence $\HH^i(\mathbb{Y}, \mathfrak{F})$ is $A$-flat for all $i \ge 0$.
	This settles part (i) of the proposition.
	Then part (ii) follows from (i) (see, e.g., \cite[Lemma 2.6]{LOC_FIB_FULL_SCHEME}).
\end{proof}

\begin{notation}
	Given projective scheme $Y \subset \PP_\kk^r = \Proj(B)$ over a field $\kk$, we denote by $I_Y \subset B$ the corresponding saturated homogeneous ideal.
\end{notation}

\begin{lemma}
	\label{lem_cohom_CM_assP}
	Let $Y \subset \PP_\kk^r = \Proj(B)$ be a projective scheme over a field $\kk$.
	The following statements hold: 
	\begin{enumerate}[\rm (i)]
		\item For $1 \le i \le r$, the ideal $I_Y$ has an associated prime of codimension $i$ if and only if the following limit 
		$$
		\underset{\nu \to \infty}{\lim} \, \frac{\dim_\kk\left(\HH^{r-i}\left(Y, \OO_Y(-\nu)\right)\right)}{\nu^{r-i}/(r-i)!} \; \in \; \ZZ_+
		$$ 
		is a positive integer.
		
		\item $Y$ is Cohen-Macaulay and equidimensional if and only if $\HH^i(Y, \OO_Y(-\nu)) = 0$ for all $i<\dim(Y)$ and $\nu \gg 0$.
	\end{enumerate}
\end{lemma}
\begin{proof}
	Let $\mm = [B]_+$ be the graded irrelevant ideal of $B$.
	Set $C = B/I_Y$.
	
	(i) 
	We recall the following known properties of Ext modules: 
	\begin{itemize}[-]
		\item $\codim\left(\Ext_B^i(C, B)\right) \ge i$;
		\item a prime $\pp \in \Spec(B)$ of codimension $i$ is an associated prime of $C$ if and only if it is a minimal prime of $\Ext_B^i(C, B)$;
	\end{itemize}
	see, e.g., \cite[Theorem 1.1]{EHV}.
	Since $\delta_i = \dim\left(\Ext_B^{i}(C, B)\right) \le r+1-i$, we consider the usual truncated notion of multiplicity 
	$$
	e_{r+1-i}\left(\Ext_B^{i}(C, B)\right) \, = \, \underset{\nu \to \infty}{\lim} \, \frac{\dim_\kk\left(\left[\Ext_B^{i}(C, B)\right]_\nu\right)}{\nu^{r-i}/(r-i)!} \,=\, \begin{cases}
		e\left(\Ext_B^{i}(C, B)\right) & \text{if } \delta_i = r+1-i\\
		0 & \text{otherwise}
	\end{cases}
	$$
	(see, e.g., \cite{SERRE_LOC}).
	Then the graded local duality theorem (see \cite[Theorem 3.6.19]{BRUNS_HERZOG}) implies that 
	$$
	e_{r+1-i}\left(\Ext_B^{i}(C, B)\right) \; = \; \underset{\nu \to \infty}{\lim} \, \frac{\dim_\kk\left(\left[\HL^{r+1-i}(C)\right]_{-\nu}\right)}{\nu^{r-i}/(r-i)!}.
	$$
	There is a short exact sequence 
	$
	0 \rightarrow C \rightarrow \bigoplus_{\nu \in \ZZ} \HH^0(Y, \OO_Y(\nu)) \rightarrow \HL^1(C) \rightarrow 0
	$ 
	and an isomorphism 
	$$
	\HL^{j+1}(C) \;\cong\; \bigoplus_{\nu \in \ZZ} \HH^j(Y, \OO_Y(\nu)) \quad \text{ for all $j \ge 1$;}
	$$
	see \cite[Theorem A4.1]{EISEN_COMM}.
	As a consequence, we have the equality
	$$
	e_{r+1-i}\left(\Ext_B^{i}(C, B)\right) \; = \; \underset{\nu \to \infty}{\lim} \, \frac{\dim_\kk\left(\HH^{r-i}(Y, \OO_{Y}(-\nu))\right)}{\nu^{r-i}/(r-i)!}.
	$$
	Finally, the result follows because $e_{r+1-i}\left(\Ext_B^{i}(C, B)\right)$ is a positive integer if and only if $\Ext_B^{i}(C, B)$ has dimension $r+1-i$ (i.e., when $C$ has an associated prime of codimension $i$).
%
 	
	(ii) This is the content of the proof of \cite[Theorem III.7.6(b)]{HARTSHORNE}. 
	Alternatively, we can argue as follows.
	Let $y \in Y$ be a point and $\pp \in \Spec(B)$ the corresponding prime in $B$. 
	By definition we have that $\OO_{Y, y} = \big\lbrace \frac{f}{g} \mid f, g \in C \text{ homogeneous elements}, \,g \not\in \pp C,\, \deg(f) = \deg(g) \big\rbrace$.
	Let $C_{(\pp)}$ be the localization of $C$ at all the homogeneous elements not in $\pp$ (we follow the notation of \cite[\S 1.5]{BRUNS_HERZOG}); this is an ${}^*$local ring ring with ${}^*$maximal ideal $\pp C_{(\pp)}$.
	By \cite[Exercise 2.1.27]{BRUNS_HERZOG}, $C_{(\pp)}$ is Cohen-Macaulay if and only if $C_\pp$ is.
	Let $w \in [C]_1$ be a linear form not in $\pp C$.
	Notice that we have the equality $C_{(\pp)} = \OO_{Y, y}[w,w^{-1}]$ and that $w$ can be seen as an indeterminate over $\OO_{Y, y}$.
	Therefore, $\OO_{Y, y}$ is Cohen-Macaulay if and only if $C_\pp$ is Cohen-Macaulay.
	
	The same arguments in the proof of part (i) imply that $\HH^i(Y, \OO_Y(-\nu)) = 0$ for all $i<\dim(Y), \nu \gg 0$ if and only if $\Supp_B\left(\Ext_B^{r-i}(C, B)\right) \subseteq \{\mm\}$ for all $i < \dim(Y)$.
	The latter condition is equivalent to the following two conditions: 
	\begin{enumerate}[(a)]
		\item  all minimal primes of $I_Y$ have codimension equal to $r - \dim(Y)$;
		\item $C_\pp$ is Cohen-Macaulay for all $\pp \in V(I_Y) \setminus \{\mm\} \subset \Spec(B)$;
	\end{enumerate} indeed, this follows from the local duality theorem and the depth sensitivity of local cohomology (see \cite[Theorem 3.5.7]{BRUNS_HERZOG}).
	So, the proof of part (ii) is complete.
\end{proof}

By combining the above results we obtain the following theorem. 

\begin{theorem}
	\label{thm_fib_full_crit}
	Let $(A, \nnn, \kk)$ be a Noetherian local ring with residue field $\kk$.
	Let $\mathbb{Y} \subset \PP_A^r$ be a closed subscheme and suppose that $\mathbb{Y}$ is flat over $A$.
	Let $Y = \mathbb{Y} \times_{\Spec(A)} \Spec(\kk) \subset \PP_\kk^r$ be the special fiber.
	Let $\qqq \in \Spec(A)$ and $Z = \mathbb{Y} \times_{\Spec(A)} \Spec(\kappa(\qqq)) \subset \PP_{\kappa(\qqq)}^r$.
	
	If \,$\Hom_\kk\left(\HH^i(Y, \OO_Y(-\nu)),\, \HH^{i+1}(Y, \OO_Y(-\nu))\right) = 0$ for all $i \ge 0$ and $\nu \gg 0$, then the following statements hold: 
	\begin{enumerate}[\rm (i)]
		\item $Y$ is Cohen-Macaulay and equidimensional if and only if $Z$ is both.
		\item For $1 \le i \le r$, $I_Y$ has an associated prime of codimension $i$ if and only if $I_Z$ has one.
	\end{enumerate}
\end{theorem}
\begin{proof}
	By \autoref{prop_criterion_flat_cohom}, we obtain the following equality 
	$$
	\dim_\kk\left(\HH^i(Y, \OO_{Y}(-\nu))\right) \; = \; \dim_{\kappa(\qqq)}\left(\HH^i(Z, \OO_{Z}(-\nu))\right) 
	$$
	for all $i \ge 0$ and $\nu \gg 0$.
	Then the statements of the theorem follow from \autoref{lem_cohom_CM_assP}.
\end{proof}

\subsection{Properties of multiplicity-free varieties}
Throughout this subsection we shall use the following setup.

\begin{setup}
	\label{setup_mult_free}
	Let $\kk$ be a field and $S = \kk\left[x_{i,j} \mid 1 \le i \le p, 0 \le j \le m_i\right]$ be a standard $\NN^p$-graded polynomial ring over $\kk$, such that $\multProj(S) = \PP = \PP_\kk^{m_1} \times_\kk \cdots \times_\kk \PP_\kk^{m_p}$.
	Let $X = \multProj(R) \subset \PP$ be an integral closed subscheme and $R$ be an $\NN^p$-graded quotient ring of $S$ that is a domain.
\end{setup}

Our next goal is to obtain an alternative proof of the following beautiful result.

\begin{theorem}[{Brion \cite{BRION_MULT_FREE}}]
	\label{thm_Brion_mult_free}
	Assume \autoref{setup_mult_free}.
	If $X \subset \PP$ is multiplicity-free, then: 
	\begin{enumerate}[\rm (i)]
		\item $X$ is arithmetically Cohen-Macaulay.
		\item $X$ is arithmetically normal.
		\item {\rm(}$\kk$ infinite{\rm)} There is a flat degeneration of $X$ to the following reduced union of multiprojective spaces
		$$
		H \;=\; \bigcup_{\bn = (n_1,\ldots,n_p) \,\in\, \msupp_\PP(X)} 	\PP_\kk^{n_1} \times_\kk \cdots \times_\kk \PP_\kk^{n_p}  \;\, \subset \;\, \PP = \PP_{\kk}^{m_1} \times_\kk \cdots \times_\kk \PP_{\kk}^{m_p}
		$$		
	\end{enumerate}
	where $\PP_\kk^{n_i} = \Proj\left(\kk[x_{i,m_i-n_i}, \ldots,x_{i,m_i}]\right) \subset \PP_{\kk}^{m_i} = \Proj\left(\kk[x_{i,0}, \ldots,x_{i,m_i}]\right)$ only uses the last $n_i+1$ coordinates.
\end{theorem}

The proof of this theorem is divided in several steps.
We first consider the case of one-dimensional multiplicity-free varieties.

\begin{lemma}
	\label{lem_dim_one_mult_free}
	Suppose that $\dim(X) = 1$ and that $X \subset \PP$ is multiplicity-free.
	Then $X$ is isomorphic to the image of a diagonal morphism $\PP_\kk^1 \rightarrow \left(\PP_\kk^1\right)^e \subset \PP$ for some $0 < e \le p$.
\end{lemma}
\begin{proof}
	The basic idea is to express $X$ as a blow-up (see \cite[Theorem II.7.17]{HARTSHORNE}).
	We have that $\Pi_i(X) \cong \Proj\left(R_{(i)}\right)$ with $R_{(i)} = \bigoplus_{k \ge 0} {\left[R\right]}_{k\cdot \ee_i}$ (see \autoref{nota_projections}) and the natural surjection $R_{(1)} \otimes_\kk \cdots \otimes_\kk R_{(p)} \surjects R$ of standard $\NN^p$-graded $\kk$-algebras, and so as a consequence we obtain that the natural morphism $X \rightarrow \Pi_1(X) \times_\kk \cdots \times_\kk \Pi_p(X)$ is a closed immersion.
	By \autoref{cor_mult_free}, for each $1 \le i \le p$, we have that $\Pi_i(X)$ equals either $\PP_\kk^1$ or $\Spec(\kk)$.
	Since  $X \hookrightarrow \Pi_1(X) \times_\kk \cdots \times_\kk \Pi_p(X)$ is a closed immersion, we may assume that $X \subset \left(\PP_\kk^1\right)^e \subset \PP$ for some $0 < e \le p$.
	Let $\alpha : X \subset \left(\PP_\kk^1\right)^e \rightarrow \PP_\kk^1$ be the projection into the last component and $\beta : X \subset \left(\PP_\kk^1\right)^e \rightarrow \left(\PP_\kk^1\right)^{e-1}$ into the first $e-1$ components.
	Due to \autoref{lem_proj_same_dim}, $\alpha$ is a birational morphism, and so there exists a dense open subset $\mathcal{U} \subset \PP_\kk^1$ such that $\alpha_{|\mathcal{V}} : \mathcal{V} = \alpha^{-1}(\mathcal{U}) \rightarrow \mathcal{U}$ is an isomorphism.
	Then the composition $\beta \circ \alpha_{|\mathcal{V}}^{-1}$ gives a morphism $\mathcal{U} \rightarrow \left(\PP_\kk^1\right)^{e-1}$ and $X$ can be realized as its graph.
	That is, we have the following commutative diagram:
		\begin{equation*}		
		\begin{tikzpicture}[baseline=(current  bounding  box.center)]
			\matrix (m) [matrix of math nodes,row sep=2.3em,column sep=5em,minimum width=2em, text height=1.5ex, text depth=0.25ex]
			{
				X \subset \left(\PP_\kk^1\right)^e &    \\
				\PP_\kk^1 & \left(\PP_\kk^1\right)^{e-1}. \\
			};
			\path[-stealth]
			(m-1-1) edge node [left]  {$\alpha$} (m-2-1)
			(m-1-1) edge node [above]  {$\beta$} (m-2-2)
			(m-2-1) edge node [above]  {$\mathcal{F}$} (m-2-2)
			;
		\end{tikzpicture}	
	\end{equation*}
	Set $\PP_\kk^1 = \Proj(B)$ with $B = \kk[z_0,z_1]$.
	The morphism $\mathcal{F}$ is given by $e-1$ pairs of homogeneous polynomials $\{f_{i,0}, f_{i,1}\} \subset B$ with $\delta_i = \deg(f_{i,0}) = \deg(f_{i,1})$ and $\gcd(f_{i,0}, f_{i,1}) = 1$.
	Let $Y \subset \PP' = \left(\PP_\kk^1\right)^{e-1}$ be the image of $\FF$.
	As a consequence of \autoref{cor_mult_free}, $Y$ is multiplicity-free.
	The degree formula of \cite[Corollary 3.9]{cid2021study} yields the equalities $\delta_i = \deg_{\PP'}^{\ee_i'}(Y) \deg(\mathcal{F}) = 1$ for all $1 \le i \le e-1$, where $\ee_i' \in \NN^{e-1}$ denotes the standard basis vector.
	Therefore, up to a $\kk$-linear change of coordinates we may assume that $f_{i,0} = z_0$ and $f_{i,1} = z_1$.
	This shows that $X$ is isomorphic to the image of the diagonal morphism $\PP_\kk^1 \rightarrow \left(\PP_\kk^1\right)^e \subset \PP$.
\end{proof}

The next proposition deals with the Hilbert polynomial of $X$.
This is the main step in the proof of \autoref{thm_Brion_mult_free}, and it is the place where we utilize the tools developed in \autoref{subsect_cohom_assP_degen}. 
		
\begin{proposition}
	\label{prop_eq_Hilb_poly}
	Under the notation and assumptions of \autoref{thm_Brion_mult_free}, we have $P_X(\ttt) = P_H(\ttt)$.
\end{proposition}	
\begin{proof}
	We proceed by induction on $\dim(X)$.
	
	If $\dim(X) = 0$, then the result is clear since we have $P_X(\ttt) = 1 = P_H(\ttt)$.
	
	If $\dim(X) = 1$, then $X$ is isomorphic to the image of a diagonal morphism $\PP_\kk^1 \rightarrow \left(\PP_\kk^1\right)^e \subset \PP$ by \autoref{lem_dim_one_mult_free}.
	Thus, we get $P_X(\ttt) = 1 + \sum_{ i \,\mid\, \ee_i \in \msupp_{\PP}(X)} t_i =  P_H(\ttt)$.
	
	Suppose that $d = \dim(X) \ge 2$.
	Here we use the conventions of \autoref{nota_generic}.	
	Write $R = S/P$ and notice that $P_\LL := P \otimes_\kk \LL \subset S_\LL$ is again a prime ideal (see \autoref{rem_infinite_field}).
	One way of relating $X$ and $H$ is by using the multigraded generic initial ideal (see \autoref{sect_multdeg_primes}).
	Let $>$ be a monomial order on $S$ that satisfies $y_0 > \cdots > y_\ell$, where $y_j = x_{p,j}$.
	Let $I = \gin(P_\LL) \subset S_\LL$.
	Let $P_1,\ldots,P_h$ be the minimal primes of $I$ and $Q_1,\ldots,Q_h$ be the corresponding minimal primary components.
	Since $X$ is assumed to be multiplicity-free, \autoref{rem_mdeg_mult_min_primes} yields that $Q_k=P_k$ for all $1 \le k \le h$.
	Set $J = Q_1 \cap \cdots \cap Q_h = \sqrt{I} \subset S_\LL$.
	By construction we have $H_\LL := H \otimes_\kk \LL \cong \multProj(S_\LL/J)$.
	Let $G := \multProj(S_\LL/I)$.
	There exists a one-parameter flat family $\mathfrak{X} \subset \PP_\LL \times_\LL \bbA_\LL^1$ such that $X_\LL$ is the general fiber and $G$ is the special fiber (see \cite[\S 15.8]{EISEN_COMM}).
	Then we write
	$$
	P_X(\ttt) \,=\, P_G(\ttt) \,=\, \sum_{\substack{\bn \in \NN^p \\ |\bn| \le d}} e(\bn) \binom{t_1+n_1}{n_1}\cdots \binom{t_p+n_p}{n_p}
	$$
	and 
	$$
	P_H(\ttt) \,=\, \sum_{\substack{\bn \in \NN^p \\ |\bn| \le d}} h(\bn) \binom{t_1+n_1}{n_1}\cdots \binom{t_p+n_p}{n_p}.
	$$
	Since $G_{\text{red}} = H_\LL$, it follows that $\dim\left(\Pi_\fJ(G)\right)  = \dim\left(\Pi_\fJ(H)\right)$ for all $\fJ \subset [p]$.
	From \cite[Proposition 3.1]{POSITIVITY}, we have  that $\dim\left(\Pi_i(X)\right)$ equals the degree of $P_X(\ttt)$ in the variable $t_i$; and the same holds for $G$ and $H$.
	Suppose that $\dim\left(\Pi_p(X)\right) \ge 1$.
	We have the equalities 
	$$
	P_{X_\LL \cap V(y)}(\ttt) \,=\, P_{G \cap V(y)}(\ttt) \,=\, \sum_{\substack{\bn \in \NN^p \\ |\bn| \le d, n_p \ge 1}} e(\bn) \binom{t_1+n_1}{n_1}\cdots\binom{t_{p-1}+n_{p-1}}{n_{p-1}} \binom{t_p+n_p-1}{n_p-1}
	$$
	and 
	$$
	P_{H_\LL \cap V(y)}(\ttt) \,=\, \sum_{\substack{\bn \in \NN^p \\ |\bn| \le d, n_p \ge 1}} h(\bn) \binom{t_1+n_1}{n_1}\cdots\binom{t_{p-1}+n_{p-1}}{n_{p-1}} \binom{t_p+n_p-1}{n_p-1}.
	$$
	As no monomial of $J$ involves the variable $y_\ell$, we obtain $J + (y) = \bigcap_{k =1}^h \left(Q_k + (y)\right)$.
	This implies that 
	$$
	H_\LL \cap V(y) \;=\; \bigcup_{\substack{\bn = (n_1,\ldots,n_p) \,\in\, \msupp_\PP(X)\\ n_p \ge 1}} 	\PP_\LL^{n_1} \times_\LL \cdots \times_\LL \PP_\LL^{n_{p-1}} \times_\LL \PP_\LL^{n_p-1}  \; \subset \; \PP_\LL.	
	$$
	On the other hand, $X_\LL \cap V(y)$ is an integral scheme, and \autoref{thm_dim_proj} and \cite[Theorem A]{POSITIVITY} yield that $\msupp_{\PP_\LL}\left(X_\LL \cap V(y)\right) = \lbrace \bn - \ee_p \mid  \bn \in \msupp_\PP(X) \text{ and } n_p \ge 1 \rbrace$.
	Therefore, the inductive hypothesis gives the equality $P_{X_\LL \cap V(y)}(\ttt) = P_{H_\LL \cap V(y)}(\ttt)$, and so we obtain that $e(\bn) = h(\bn)$ for all $n_p \ge 1$.
	
	For $1 \le i < p$, we can repeat the above arguments and consider the $i$-th component of $\PP = \PP_\kk^{m_1}\times_\kk \cdots \times_\kk \PP_\kk^{m_p}$.
	So, we conclude that $e(\bn) = h(\bn)$ for all $\bn \neq \mathbf{0}$.
	Thus $c = P_X(\ttt) - P_H(\ttt)$ is a constant polynomial.
	
	We have a short exact sequence $0 \rightarrow J/I \rightarrow S_\LL/I \rightarrow S_\LL/J \rightarrow 0$. 
	Hence $P_{J/I}(\ttt) = c$ and the corresponding coherent sheaf $\mathfrak{K} = (J/I)^\sim$ on $G$ is supported on dimension zero.
	By \autoref{thm_sqrt_gin_CM}, $S_\LL/J$ is Cohen-Macaulay, and thus $H_\LL$ is Cohen-Macaulay.
	
	Let $\mathfrak{s} : \PP_\LL = \PP_\LL^{m_1} \times_\LL \cdots \times_\LL \PP_\LL^{m_p} \rightarrow \PP_\LL^r$ be the Segre embedding, where $r = (m_1+1)\cdots(m_p+1)-1$.
	Let $Z = \mathfrak{s}(X_\LL)$, $Y = \mathfrak{s}(G)$ and $Y' = \mathfrak{s}(H_\LL)$.
	Let $\mathfrak{S} = \mathfrak{s} \times_\LL \bbA_\LL^1 : \PP_\LL \times_\LL \bbA_\LL^1 \rightarrow \PP_\LL^r \times_\LL \bbA_\LL^1$, and notice that $\mathbb{Y} = \mathfrak{S}(\mathfrak{X}) \subset \PP_\LL^r \times_\LL \bbA_\LL^1$ is a one-parameter flat family with general fiber $Z \subset \PP_\LL^r$ and special fiber $Y \subset \PP_\LL^r$.
	We get a short exact sequence $0 \rightarrow \mathfrak{s}_*\mathfrak{K} \rightarrow \OO_Y \rightarrow \OO_{Y'} \rightarrow 0$.
	As $\mathfrak{s}_*\mathfrak{K}$ is supported on dimension zero, we have $\HH^i(Y, (\mathfrak{
	s}_*\mathfrak{K})(\nu)) = 0$ for all $i \ge 1, \nu \in \ZZ$.
	On the other hand, \autoref{lem_cohom_CM_assP}(ii) yields $\HH^i(Y', \OO_{Y'}(-\nu)) = 0$ for all $i<d, \nu \gg 0$.
	From the induced long exact sequence in cohomology
	$$
	\cdots \,\rightarrow\, \HH^{i}(Y, (\mathfrak{
		s}_*\mathfrak{K})(-\nu)) \,\rightarrow\, \HH^i(Y, \OO_Y(-\nu)) \,\rightarrow\, \HH^i(Y', \OO_{Y'}(-\nu)) \,\rightarrow\, \cdots
	$$
	we obtain that $\HH^i(Y, \OO_{Y}(-\nu)) = 0$ for all $1\le i<d, \nu \gg 0$.
	Since $\dim(Y) = d \ge 2$, it follows that the condition of \autoref{thm_fib_full_crit} holds for our current $Y$.
	The ideal $I_Z$ does not have an associated prime of codimension $r$ because $Z$ is an integral scheme, and so \autoref{thm_fib_full_crit}(ii) implies that $I_Y$ does not have an associated prime of codimension $r$.
	Therefore $\mathfrak{s}_*\mathfrak{K} = 0$ and $Y = Y'$.
	
	Finally, we have shown that $c = P_X(\ttt) - P_H(\ttt) = 0$, and this concludes the proof of the proposition. 
\end{proof}

We can now complete our proof of Brion's result.

\begin{proof}[Proof of \autoref{thm_Brion_mult_free}]
	Let $P \subset S$ be a prime ideal with $X = \multProj(S/P)$.
	Let $>$ be a monomial order on $S$ and assume that $\kk$ is infinite (see \autoref{rem_infinite_field}).
	As in \autoref{rem_relev_prim} and the beginning of the proof of \autoref{thm_proj_gin}, we may adjoin new variables to $S$ in such a way that all the associated primes of $\gin(P)$ are relevant.
	Let $I = \gin(P)$ and $J = \sqrt{\gin(P)}$.
	Since $X$ is multiplicity-free, \autoref{rem_mdeg_mult_min_primes} implies that $J$ is the intersection of the minimal primary components of $I$.
	By \autoref{prop_eq_Hilb_poly}, we have the equalities $P_{S/I}(\ttt) = P_{S/J}(\ttt) = P_{S/P}(\ttt)$, and since all the associated primes of $I$ are assumed to be relevant, we obtain $I = J$.
	
	The result of part (iii) now follows at once (see \cite[\S 15.8]{EISEN_COMM}).
	Moreover, we have proved that the multigraded Hilbert function of $S/P$ equals the one of $S/J$.
	From \autoref{thm_sqrt_gin_CM} we have that $S/J$ is Cohen-Macaulay and as a consequence $S/P$ also is.
	This settles part (i).
	
	Let $\overline{R}$ be the integral closure of $R=S/P$ in its field of fractions $\Quot(R)$.
	We have that $\overline{R}$ is a finitely generated $\NN^p$-graded $R$-module (see \cite[\S 2.3]{SwHu}).
	Write $\overline{R} = R[w_1,\ldots,w_b] \subset \Quot(R)$ where $w_1,\ldots,w_b$ is a minimal set of homogeneous generators of $\overline{R}$ as an $R$-algebra. Suppose by contradiction that each $w_l \not\in R$.
	Let $T = S[z_1,\ldots,z_b]$ be an $\NN^p$-graded polynomial ring over $S$ with a surjective $S$-algebra homomorphism $\psi : T \twoheadrightarrow \overline{R}, \, z_l \rightarrow w_l$.
	Since $\rank_R\left(\overline{R}\right)=1$, the associativity formula for mixed multiplicities yields the equality $e(\bn; \overline{R}) = e(\bn; R)$ for all $\bn \in \NN^p$ with $|\bn| = \dim(X)$.
	Hence we consider the multiplicity-free prime ideal $Q = \Ker\left(\psi\right) \subset T$.
	Choose a lexicographical monomial order $>'$ on $T$ such that $z_l >' x_{i,j}$ for all $1 \le l \le b, 1 \le i \le p, 0 \le j \le m_i$.
	The assumption $w_l \not\in R$ would imply that a monomial of the form $z_l^{g_l}$ with $g_l \ge 2$ is among the minimal generators of $\init_{>'}(Q)$. 
	However, this would contradict the fact that $\init_{>'}(Q)$ is radical by \autoref{thm_props_CS}(i) and \autoref{prop_sqr_free_primes}; notice that the proof of \autoref{prop_sqr_free_primes} only depends on the fact that $\gin(P)$ is radical, which we already proved.
	So, $R$ is a normal domain and the proof of the remaining part (ii) is complete.
\end{proof}

Finally, we restate Brion's theorem in a more algebraic language.

\begin{theorem}
	\label{thm_Brion}
	Assume \autoref{setup_mult_free}. 
	If $P \subset S$ is an $S$-homogeneous multiplicity-free prime ideal, then: \begin{enumerate}[\rm (i)]
		\item $S/P$ is Cohen-Macaulay.
		\item $S/P$ is a normal domain.
		\item {\rm(}$\kk$ infinite{\rm)} $\gin_>(P)$ is a radical monomial ideal for any monomial order $>$ on $S$. 
	\end{enumerate}
\end{theorem}	
\begin{proof}
	This follows from the proof of \autoref{thm_Brion_mult_free} as we can always make $P$ a relevant ideal.
\end{proof}

\section{Standardization of ideals}
	\label{sec:standardization}
	In this section, we develop a process of standardization of ideals in an \emph{arbitrary positive (possibly non-standard) multigrading}. 
	Here we use a mixture of the techniques \emph{step-by-step homogenization} introduced in \cite{McCULLOUGH_PEEVA} for singly-graded settings and \emph{standardization} introduced in \cite{DOUBLE_SCHUBERT} for certain multigraded settings.
	The following setup and construction is used throughout this section.
	
	\begin{setup}
		\label{setup_standardization}
		Let $\kk$ be a field and $R = \kk[x_1,\ldots,x_n]$ be a positively $\NN^p$-graded polynomial ring.
		For $1 \le i \le n$, let $\ell_i = |\deg(x_i)|$ be the total degree of the variable $x_i$.
		Let
		$$
		S \;= \; \kk\left[y_{i,j} \mid 1 \le i \le n \text{ and } 1 \le j \le \ell_i\right]
		$$
		be a standard $\NN^p$-graded polynomial ring such that
		$$
		\deg(x_i) \;=\; \sum_{j = 1}^{\ell_i} \deg(y_{i,j})   \quad  \text{for all \;\;$1 \le i \le n$.}
		$$
		We define the $\NN^p$-graded $\kk$-algebra homomorphism 
		\begin{equation*}
			\phi: R=\kk[\xx] \longrightarrow S = \kk[\mathbf{y}], \quad
			\phi(x_{i}) = y_{i,1}y_{i,2}\cdots y_{i,\ell_i}.
		\end{equation*}
		For an $R$-homogeneous ideal $I \subset R$, we say that the extension  $J = \phi(I) S$ is the \emph{standardization} of $I$, as $J \subset S$ is an $S$-homogeneous ideal in the standard $\NN^p$-graded polynomial ring $S$.
		Let $\ttt = \{t_1,\ldots,t_p\}$ be variables indexing the $\ZZ^p$-grading, where $t_i$ corresponds with $\ee_i  \in \ZZ^p$.
		Given a finitely generated $\ZZ^p$-graded $R$-module $M$ and a finitely generated $\ZZ^p$-graded $S$-module $N$, by a slight abuse of notation, we consider both multidegree polynomials $\mathcal{C}(M;\ttt)$ and $\mathcal{C}(N;\ttt)$ as elements of the same polynomial ring $\ZZ[\ttt]=\ZZ[t_1,\ldots,t_p]$.
	\end{setup}
	
	The following theorem contains some of the basic and desirable properties that the standardization process  satisfies.
	
	\begin{theorem}
		\label{thm_std}
		Assume \autoref{setup_standardization}.
		Let $I \subset R$ be an $R$-homogeneous ideal and $J = \phi(I)S$ be its standardization. 
		Then the following statements hold:
		\begin{enumerate}[\rm (i)]
			\item $\codim(I) = \codim(J)$.
			\item $I \subset R$ and $J \subset S$ have the same $\NN^p$-graded Betti numbers. 
			\item $\mathcal{K}(R/I;\ttt) = \mathcal{K}(S/J;\ttt)$ and  $\mathcal{C}(R/I;\ttt) = \mathcal{C}(S/J;\ttt)$.
			\item $R/I$ is a Cohen-Macaulay ring if and only if $S/J$ is a Cohen-Macaulay ring.
			\item Let $>$ be a monomial order on $R$ and $>'$ be a monomial order on $S$ which is compatible with $\phi$ {\rm(}i.e.,~if $f,g \in R$ with $f > g$, then $\phi(f) >' \phi(g)${\rm)}.
			Then $\init_{>'}(J) = \phi(\init_{>}(I))S$.
			\item If $I \subset R$ is a prime ideal and it does not contain any variable, then $J \subset S$ is also a prime ideal.
		\end{enumerate}
	\end{theorem}
	\begin{proof}
		Let $T$ be the polynomial ring $T=\kk[\xx,\yy] \cong R \otimes_\kk S$ with its natural $\NN^p$-grading induced from the ones of $R$ and $S$. 
		We see $R$ and $S$ as subrings of $T$.
		We consider the quotient ring $T/IT$ and notice that $\{x_{i} - \prod_{j=1}^{\ell_i} y_{i,j}\}_{1\le i \le n}$ is a regular sequence of homogeneous elements on $T/IT$.
		We also have the following natural isomorphism 
		$$
		\frac{T}{IT + \left(\{x_{i} - \prod_{j=1}^{\ell_i} y_{i,j}\}_{1\le i \le n}\right)} \;\cong\; S/J.
		$$
		As the natural inclusion $R \hookrightarrow T$ is a polynomial extension, we have that $\dim(T/IT) = \dim(R/I) + \dim(S) = \dim(R) + \dim(S) - \codim(I)$ and that $T/IT$ is Cohen-Macaulay if and only if $R/I$ is Cohen-Macaulay.
		So, by cutting out with the regular sequence described above, we obtain that 
		$		
		\dim(S/J) = \dim(T/IT) - \dim(R) = \dim(S) - \codim(I)
		$
		and that $S/J$ is Cohen-Macaulay if and only if $T/IT$ is Cohen-Macaulay.
		This completes the proofs of parts (i) and (iv).
		
		Let $F_\bullet : \cdots  \xrightarrow{f_2} F_1 \xrightarrow{f_1} F_0$ be a $\ZZ^p$-graded free $R$-resolution of $R/I$.
		Since $\{x_{i} - \prod_{j=1}^{\ell_i} y_{i,j}\}_{1\le i \le n}$ is a regular sequence on both $T$ and $T/IT$, it follows that 
		$$
		\Tor_k^T\Big(T/IT, T/\Big(\{x_{i} - \prod_{j=1}^{\ell_i} y_{i,j}\}_{1\le i \le n}\Big)\Big) = 0 \quad \text{ for all $k > 0$,}
		$$ and so $G_\bullet = F_\bullet \otimes_R T/(\{x_{i} - \prod_{j=1}^{\ell_i} y_{i,j}\}_{1\le i \le n})$ provides (up to isomorphism) a $\ZZ^p$-graded free $S$-resolution of $S/J$.
		The identification of $G_\bullet$ as a resolution of $S$-modules is the same as $\phi(F_\bullet)$ (more precisely,  $G_\bullet$ has the same Betti numbers as $F_\bullet$ and the $i$-th differential matrix of $G_\bullet$ is given by the substitution $\phi(f_i)$ of the differential matrix $f_i$).
		We obtained the result of part (ii).
		Therefore,  by definition, we have the equalities $\mathcal{K}(R/I;\ttt) = \mathcal{K}(S/J;\ttt)$ and $\mathcal{C}(R/I;\ttt) = \mathcal{C}(S/J;\ttt)$ that settle part (iii).
		
		To show part (v) we can use Buchberger's algorithm (see, e.g., \cite[Chapter 15]{EISEN_COMM}). 
		Indeed, we can perform essentially the same steps of the algorithm in a set of generators of $I$ and the corresponding set of generators for $J$; for instance, for any two polynomials $f, g \in R$ we have the following relation of $S$-polynomials $S(\phi(f), \phi(g)) = \phi(S(f,g))$.
		
		Lastly, we concentrate on the proof of part (vi).
		Suppose that $I \subset R$ is a prime ideal not containing any variable.
		It then follows that $T/IT$ is a domain and that $\{y_{1,2}\cdots y_{1,\ell_1},\, x_1\}$ is a regular sequence on $T/IT$, and so \cite[Exercise 10.4]{EISEN_COMM} implies that  $T/(IT, x_1 - \prod_{j=1}^{\ell_1}y_{1,j})$ is a also a domain. 
		Finally, by repeating iteratively this argument we obtain that $
		T/ \big(IT + \big(\{x_{i} - \prod_{j=1}^{\ell_i} y_{i,j}\}_{1\le i \le n}\big)\big) \cong S/J
		$
		is a domain.  
		This concludes the proof of part (vi).
	\end{proof}

	\begin{remark}
		Without too many changes the above arguments could be used to define the \emph{standardization of a module}.
		In that case, one could proceed by standardizing a presentation matrix of a module.
	\end{remark}

	A direct known consequence of the above theorem is the following remark.
	
	\begin{remark}
		\label{rem_posit_non_std}
		For any $R$-homogeneous ideal $I \subset R$ the multidegree polynomial $\mathcal{C}(R/I; \ttt)$ has non-negative coefficients.
		Indeed, it follows from \autoref{thm_std}(iii) and the (well-known) fact that multidegrees are non-negative in a standard multigraded setting.
	\end{remark}

	The following theorem shows that the support of the multidegree polynomial is a discrete polymatroid for prime ideals in a polynomial ring with arbitrary positive multigrading.
	It is a consequence of \cite[Theorem A]{POSITIVITY}.
	
	\begin{theorem}
		\label{thm_polymatroid_multdeg}
		Assume \autoref{setup_standardization}.
		Let $P \subset R$ be an $R$-homogeneous prime ideal.
		Then the support of the multidegree polynomial  $\mathcal{C}(R/P;\ttt)$ is a discrete polymatroid.
	\end{theorem}
	\begin{proof}
		Let $\mathcal{L} = \{ i \mid x_i \in P \}$ be the set of indices such that the corresponding variables belong to $P$.
		We consider the polynomial ring
		$
		R' =  \kk[x_{i} \mid i \notin \mathcal{L}] \subset R.
		$
		Let $P' \subset R'$ be the (unique) ideal that satisfies the condition $P = P'R + \left(x_i \mid i \in \mathcal{L}\right)$.
		By construction $P' \subset R'$ does contain any variable.
		Since $R/P \cong R'/P'$, it follows that $P'$ is also a prime ideal.
		As a consequence of \autoref{thm_std}(vi) the standardization $Q = \phi(P'R) \subset S$ is a prime ideal.
		
		Since the variables $x_{i}$ with indices in $\mathcal{L}$ form a regular sequence on $R/P'R$, we obtain the equation 
		$$
			\mathcal{C}(R/P;\ttt) \;=\; \prod_{i \in \mathcal{L}} \langle \deg(x_i), \ttt \rangle \,\cdot\, \mathcal{C}(R/P'R;\ttt) 
		$$
		where $\langle \deg(x_i), \ttt \rangle = a_{i,1}t_1 + \cdots + a_{i,p}t_p \in \NN[\ttt]$ after writing $(a_{i,1}, \ldots,a_{i,p}) = \deg(x_i) \in \NN^p$
		(see \autoref{rem_multdeg_facts}(ii)).
		From \autoref{thm_pos_multdeg} we have that the support of $\mathcal{C}(S/Q; \ttt)$ is a discrete polymatroid, and then \autoref{thm_std}(iii) implies that the support of $\mathcal{C}(R/P'R; \ttt)$ is also a discrete polymatroid.
		On the other hand, it is clear that the support of each linear polynomial $\langle \deg(x_i), \ttt \rangle$ is a discrete polymatroid.
		Finally, \autoref{rem_basic_polymatroids}(ii) yields that the support of $\mathcal{C}(R/P;\ttt)$ is a discrete polymatroid.
	\end{proof}

	\begin{corollary}
		\label{cor_CM_gin_std}
		Assume \autoref{setup_standardization} with $\kk$ an infinite field.
		Let $P \subset R$ be an $R$-homogeneous prime ideal, and $Q = \phi(P)S$ be its standardization.
		Then $\sqrt{\gin(Q)}$ is a Cohen-Macaulay ideal.
	\end{corollary}
	\begin{proof}
		By \autoref{thm_polymatroid_multdeg} and \autoref{thm_std}(iii), the support of $\mathcal{C}(S/Q;\ttt)$ is a discrete polymatroid.
		One may notice that the proof of \autoref{thm_sqrt_gin_CM} only depends on the fact that the support of the multidegree polynomial of the ideal is a discrete polymatroid. 
		Therefore, by using the same arguments applied to $Q$ (which may not be prime), we conclude that $\sqrt{\gin(Q)}$ is a Cohen-Macaulay ideal.
	\end{proof}
	
	\subsection{Cartwright-Sturmfels ideals in positive multigradings}
	\label{subsect_CS_ideals}
	
	The purpose of this subsection is to extend the notion of Cartwright-Sturmfels ideals to arbitrary positive multigradings. 
	This family of ideals was defined and studied in a series of papers \cite{CDNG_CS_IDEALS,CDNG_GIN,CDNG_GRAPH,CDNG_MINORS,conca2022radical} over a standard multigraded setting.

	\begin{definition}
		\label{def_new_CS}
		An $R$-homogeneous ideal $I \subset R$ is said to be \emph{Cartwright-Sturmfels} (CS for short) if there exists an $S$-homogeneous radical Borel-fixed ideal $K \subset S$ such that $\mathcal{K}(R/I;\ttt) = \mathcal{K}(S/K;\ttt)$.
	\end{definition}

	The first basic (yet important) observation that one can make is the following lemma.
	
	\begin{lemma}
		\label{lem_equiv_def_CS}
		Let $I \subset R$ be an $R$-homogeneous ideal and $J = \phi(I)S$ be its standardization.
		Then, $I$ is {\rm CS} {\rm(}in the sense of \autoref{def_new_CS}{\rm)} if and only if $J$ is {\rm CS} {\rm(}in the sense of \cite{CDNG_CS_IDEALS}{\rm)}.
	\end{lemma}
	\begin{proof}
		By \autoref{thm_std}(iii), we have $\mathcal{K}(R/I;\ttt) = \mathcal{K}(S/J;\ttt)$. 
		Thus any $S$-homogeneous ideal $K \subset S$ has the same $\mathcal{K}$-polynomial as $I \subset R$ if and only if it has the same multigraded Hilbert function as $J \subset S$.
		So, the equivalence is clear because $J$ is CS when there is a radical Borel-fixed ideal with the same multigraded Hilbert function as $J$.
	\end{proof}

	\begin{proposition}
		\label{prop_sqr_free_primes}
		Let $P \subset R$ be a multiplicity-free prime ideal, then $P$ is $\CS$.
	\end{proposition}
	\begin{proof}
		As in the proof of \autoref{thm_polymatroid_multdeg}, we may get rid of the variables that belong to $P$.
		Let $\mathcal{L} = \{ i \mid x_i \in P \}$ be the set that indexes the variables belonging to $P$.
		Write $P = P' + \left(x_i \mid i \in \mathcal{L}\right)$ with $P' \subset R$ only involving the variables not in $\mathcal{L}$.
		Let $T = R\left[z_i \mid i \in \mathcal{L}\right]$ be a positively $\NN^p$-graded polynomial ring extending the grading of $R$ and with $\deg(z_i) = \deg(x_i)$.
		Let $\mathfrak{P} = P'T + \left(x_i - z_i \mid i \in \mathcal{L}\right) \subset T$.
		Notice that $\mathfrak{P}$ is a prime ideal containing no variable and that $\mathcal{K}(T/\mathfrak{P}; \ttt) = \mathcal{K}(R/P; \ttt)$.
		Let $W = S\left[w_{i,j} \mid i \in \mathcal{L}, 1 \le j \le \ell_i \right]$ be a standard $\NN^p$-graded polynomial ring where we consider the corresponding standardization $\mathfrak{Q} \subset W$ of $\mathfrak{P} \subset T$.
		From \autoref{thm_std}(iii)(vi) the standardization $\mathfrak{Q} \subset W$ is a multiplicity-free prime ideal.
		Then \autoref{thm_Brion}(iii) implies that $\mathfrak{Q}$ is CS (as we may extend the field $\kk$; see \autoref{rem_infinite_field}).
		Let $Q \subset S$ be the standardization of $P \subset R$.
		Since ideal extension $QW \subset W$ has the same multigraded Hilbert function as $\mathfrak{Q} \subset W$, it follows that $QW$ is CS, and so \cite[Proposition 2.7]{CDNG_CS_IDEALS} implies that $Q \subset S$ is also CS.
		Finally, $P \subset R$ is CS by \autoref{lem_equiv_def_CS}.
	\end{proof}

	The theorem below lists some important basic properties of CS ideals in $R$.
	
	\begin{theorem}
		\label{thm_props_CS}
		Assume \autoref{setup_standardization}.
		Let $I \subset R$ be an $R$-homogeneous {\rm CS} ideal.
		Then the following statements hold: 
		\begin{enumerate}[\rm (i)]
			\item $\init_{>}(I)$ is radical and {\rm CS} for any monomial order $>$ on $R$; in particular, $I$ is radical.
			\item The $\NN$-graded Castelnuovo-Mumford regularity $\reg(I)$ is bounded from above by $p$.
			\item $I$ is a multiplicity-free ideal.
			\item If $P \subset R$ is a minimal prime of $I$, then $P$ is CS.
			\item All reduced Gr\"obner bases of $I$ consist of elements of multidegree $\le (1,\ldots,1) \in \NN^p$.
			In particular, $I$ has a universal Gr\"obner basis of elements of multidegree $\le (1,\ldots,1) \in \NN^p$.
		\end{enumerate}
	\end{theorem}
	\begin{proof}
		(i) It follows from \autoref{thm_std}(v) and \cite[Proposition 3.4(2)]{conca2022radical}.
		
		(ii) It is a consequence of \autoref{thm_std}(ii) and \cite[Proposition 3.4(3)]{conca2022radical}.
		
		(iii) It follows from \autoref{thm_std}(iii) and \cite[Proposition 3.6(1)]{conca2022radical}.
		
		(iv) Since $I$ is multiplicity-free, we conclude that $P$ also is. 
		We then obtain the result from \autoref{prop_sqr_free_primes}.
		
		(v) It follows from \autoref{thm_std}(v) and \cite[Proposition 3.4(6)]{conca2022radical}.
	\end{proof}

	Finally, we point out some interesting consequences of our work, that signal very rigid properties of certain multigraded Hilbert functions.
	To that end, we use the multigraded Hilbert scheme of Haiman and Sturmfels \cite{HAIMAN_STURMFELS}.
	For a given function $h : \NN^p \rightarrow \NN$, the multigraded Hilbert scheme $\HS_{R/\kk}^h$  parametrizes all $R$-homogeneous ideals $I \subset R$ such that $\dim_\kk\left(\left[R/I\right]_\nu\right) = h(\nu)$ for all $\nu \in \NN^p$.
	We have the following rigidity result that extends \cite[Theorem 2.1 and Corollary 2.6]{CS_PAPER}.
	
	\begin{corollary}
		\label{cor_Hilb_sch}
		Assume \autoref{setup_standardization}.
		Let $h:\NN^p \rightarrow \NN$ be a function and consider the corresponding multigraded Hilbert scheme $\HS_{R/\kk}^h$.
		If $\HS_{R/\kk}^h$ contains a $\kk$-point that corresponds to a {\rm CS} ideal and a $\kk$-point that corresponds to a prime ideal, then any $\kk$-point in $\HS_{R/\kk}^h$ corresponds to an ideal that is Cohen-Macaulay and {\rm CS}.
	\end{corollary}
\begin{proof}
	Let $[P], [H] \in \HS_{R/\kk}^h$ be $\kk$-points in $\HS_{R/\kk}^h$ such that $P \subset R$ is a prime ideal and $H \subset R$ is a CS ideal.
	Let $I \subset R$ be an $R$-homogeneous ideal such that $[I] \in \HS_{R/\kk}^h$.
	From the definition, it follows that $I$ and $P$ are both CS.
	We may assume that $\kk$ is a infinite field and keep the primeness of $P$ (see \autoref{rem_infinite_field}).
	Let $J \subset S$ and $Q \subset S$ be the standardizations of $I$ and $P$, respectively.
	Since $J$ and $Q$ are CS, both $\gin(J)$ and $\gin(Q)$ are radical, and actually are equal.
	Then $\gin(J) = \gin(Q)$ is Cohen-Macaulay by \autoref{cor_CM_gin_std}.
	Finally, $J \subset S$ is Cohen-Macaulay and \autoref{thm_std}(iv) implies that $I \subset R$ is also Cohen-Macaulay; which concludes the proof.
\end{proof}
	
\section{Examples of determinantal ideals with a fine grading}
	\label{sect_examp_det}

	The goal of this section is to present examples of multigraded ideals in  non-standard positive multigraded polynomial rings that have a multiplicity-free multidegree.
	The examples we selected are determinantal ideals.  
	
	We start with generic determinantal ideals with the finest multidegree structure. 
	Let $\kk$ be a field and set $R=\kk[x_{i,j} \mid (i,j)\in [m]\times [n] ]$ with $m\leq n$.  
	We give $R$ a  multigraded structure by setting $\deg(x_{i,j})=\ee_i\oplus \ff_j \in \NN^m\oplus \NN^n$ where $\ee_1,\dots,\ee_m$ and $\ff_1,\dots, \ff_n$ are  the canonical bases of $\NN^m$ and $\NN^n$, respectively.  
	The determinantal ideal $I_r$ of the $r$-minors of the $m \times n$ generic matrix $X = (x_{i,j})$ is then $R$-homogeneous.  
	We have: 
	
	\begin{theorem} 
		\label{maxmin}
		The  $\kk$-algebra $R/I_m$  has a multiplicity-free multidegree with respect to the  $\NN^m\oplus \NN^n$-grading.
	\end{theorem} 
	
	A combinatorial formula for the multidegree of $R/I_r$  is described in \cite[Chap.15 and 16]{miller2005combinatorial} (up to certain sign changes).  
	Indeed, by \cite[Theorem 15.40 and Corollary 16.30]{miller2005combinatorial}, the multidegree of a larger family of determinantal ideals, the Schubert determinantal ideals, is given in terms of pipe-dreams (i.e.,~special tiling of the rectangles by crosses and elbow joints). 
	So one could derive \autoref{maxmin} from the  results in \cite{miller2005combinatorial}. 
	Nevertheless, we present below a self-contained proof of \autoref{maxmin}.  
	To simplify the notation, we set 
	$$\mathcal{K}_{m,n}=\mathcal{K}(R/I_m;  t_1,\dots, t_m,s_1,\dots,s_n) 
	\quad \text{ and } \quad
	\mathcal{C}_{m,n}=\mathcal{C}(R/I_m; t_1,\dots, t_m,s_1,\dots,s_n).$$
	Indeed, we prove the following formula from which  \autoref{maxmin} follows immediately. 
	
	\begin{theorem} 
		\label{maxminC}
		$$\mathcal{C}_{m,n} \,=\, \sum_{(\mathbf{a},\mathbf{b})\in T} \ttt^\mathbf{a}\sss^\mathbf{b}$$
		with $T=\big\{( \mathbf{a},\mathbf{b})\in \NN^m\times \{0,1\}^n \, \mid \, |\mathbf{a}|+|\mathbf{b}|=n-m+1 \big\}$. 
	\end{theorem} 

	For the proof \autoref{maxminC} we denote by $H_{m,n}$ the right hand side of the equality, i.e. 
	$$H_{m,n} \,=\, \sum_{(\mathbf{a},\mathbf{b})\in T} \ttt^\mathbf{a}\sss^\mathbf{b}$$
	with $T=\big\{( \mathbf{a},\mathbf{b})\in \NN^m\times \{0,1\}^n  \, \mid \,  |\mathbf{a}|+|\mathbf{b}|=n-m+1 \big\}$.  
	Firstly we observe that $H_{m,n}$ satisfies the following recursion: 
	
	\begin{lemma} 
		\label{recursionH}
		For all $n>m>1$, one has $H_{m,n}=s_nH_{m,n-1}+t_mH_{m,n-1}+H_{m-1,n-1}$.
	\end{lemma} 
	\begin{proof} The first addend  corresponds to the sum of all the monomials in $H_{m,n}$ that contain $s_n$. 
		The second corresponds to the sum of all the  monomials in $H_{m,n}$ that do not contain $s_n$ and contain $t_m$. 
		Finally the third addend corresponds to the sum of  all the monomials in $H_{m,n}$ that do not contain $s_n$ and $t_m$. Here it is important to observe that $n-m+1=(n-1)-(m-1)+1$. 
	\end{proof} 
	Secondly we observe that also $\mathcal{C}_{m,n}$ satisfies the same recursion of \autoref{recursionH}. 
	
	\begin{lemma} 
		\label{recursionC}
		For all $n>m>1$, one has $\mathcal{C}_{m,n}=s_n\mathcal{C}_{m,n-1}+t_m\mathcal{C}_{m,n-1}+\mathcal{C}_{m-1,n-1}$.
	\end{lemma} 
	\begin{proof} 
		The $m$-minors of $X$ form a Gr\"obner basis of $I_m$ with respect to the diagonal term order. 
		Let $J_{m,n}=\init(I_m)$ be the monomial ideal generated by the main diagonals of the $m$-minors of $X$.  
		Then
		$$\mathcal{K}_{m,n}=\mathcal{K}(R/J_{m,n};  t_1,\dots, t_m,s_1,\dots,s_n)$$ and similarly for  $\mathcal{C}_{m,n}$. 
		We have a decomposition $J_{m,n}=x_{m,n}J_{m-1,n-1}+J_{m,n-1}$ from which it follows that  $J_{m,n}+(x_{m,n})=J_{m,n-1}+(x_{m,n})$ and  $J_{m,n}:(x_{m,n})=J_{m-1,n-1}$. 
		Therefore we have an induced short exact sequence
		$$0\to R/J_{m-1,n-1}\left(-\deg(x_{m,n})\right) \to R/J_{m,n}  \to R/\left(J_{m,n-1}+(x_{m,n})\right)\to 0$$
		and hence 
		\begin{equation}
			\label{eq_recurr_K_poly}
			\mathcal{K}_{m,n}=(1-t_ms_n)\mathcal{K}_{m,n-1}+t_ms_n\mathcal{K}_{m-1,n-1}.
		\end{equation}
		Replacing $\ttt$ with $\mathbf{1} - \ttt$ and $\sss$ with $\mathbf{1} - \sss$ and extracting the homogeneous component of degree $n-m+1$ we obtain
		$$\mathcal{C}_{m,n}=(t_m+s_n)\mathcal{C}_{m,n-1}+\mathcal{C}_{m-1,n-1},$$
		which concludes  the proof.
	\end{proof} 
	Now we are ready to prove \autoref{maxminC}. 
	
	\begin{proof}[Proof of  \autoref{maxminC}]  By \autoref{recursionH} and \autoref{recursionC}  the polynomials $H_{m,n}$ and $\mathcal{C}_{m,n}$  satisfy the same recursive relation.  
	Hence it suffices to check that they  agree when $m=1$ and when $m=n$. When $m=1$ the ideal $I_m$ is generated by a regular sequence of elements of degree $\ee_1+\ff_j$ with $j=1,\dots, m$. 
	In this case,  $\mathcal{C}_{1,n}=\prod_{j=1}^n(t_1+s_j)$ which clearly coincides with $H_{1,n}$. 
	When $m=n$ then $I_m$ is generated by single element of degree $\sum_{i=1}^m \ee_i+ \sum_{j=1}^m  \ff_j$, and hence $\mathcal{C}_{m,m}=\sum_{i=1}^m t_i+ \sum_{j=1}^m  s_j$ which coincides with $H_{m,m}$. 
	This concludes the proof of the theorem.
	\end{proof}

	\begin{remark}
		As a consequence of \autoref{maxmin} and \autoref{prop_sqr_free_primes}, it follows that $I_m \subset R$ is a CS ideal.
	\end{remark}

	\begin{remark}
		By inspecting the Eagon-Northcott resolution of $I_m$, we may compute that the $\NN$-graded Castelnuovo-Mumford regularity $I_m$ is given by $\reg(I_m) = m+n$. 
		Indeed, by coarsening the $\NN^m \oplus \NN^n$-grading of $R$ into an  $\NN$-grading, we obtain that each variable $x_{i,j}$ has degree $2$ and that the ideal $I_m$ is generated in degree $2m$.  
		Since the Eagon-Northcott complex has linear maps in the $x_{i,j}$'s and length $n-m$, it follows that $\reg(I_m) = n+m$.
		This shows that the upper bound of \autoref{thm_props_CS}(ii) is sharp.
	\end{remark}

	\begin{remark}
		From the recursion of $\mathcal{K}$-polynomials stated in \autoref{eq_recurr_K_poly}, we may compute the following formula for the $\mathcal{K}$-polynomial of $I_m$:
		$$
		\mathcal{K}(R/I_m; \ttt,\sss) \;=\; 1-t_1t_2 \cdots t_m \sum_{j=0}^{n-m} (-1)^j h_j(\ttt) e_{m+j}(\sss) 
		$$
		where $h_j(\ttt)$ is the complete symmetric polynomial of degree $j$ in $t_1,\ldots,t_m$ (sum of all monomials of degree $j$) and $e_j(\sss)$ is the elementary symmetric polynomial of degree $j$ in $s_1,\ldots,s_n$ (sum of all squarefree monomials of degree $j$). 
		Also, the formula for $\mathcal{K}(R/I_m; \ttt,\sss)$ may be obtained from the Eagon-Northcott resolution of $I_m$.
	\end{remark}
	
	Finally, we point out that most determinantal varieties do not have a multiplicity-free multidegree with respect to the  $\NN^m\oplus \NN^n$-grading. 
	
	\begin{remark}  
		For   $1<r<m$, the $\kk$-algebra $R/I_r$ does not have a multiplicity-free multidegree with respect to the  $\NN^m\oplus \NN^n$-grading. 
		For $2<r<m$, this happens already with 
		respect to the $\NN^m$ (row) grading and the  $\NN^n$ (column) grading individually, see \cite{conca2022radical}. 
		We have a significantly different behavior when $r=2$:  $R/I_2$ has a multiplicity free multidegree with respect to the row or column grading, but  one can check by utilizing \cite{miller2005combinatorial} or by direct computations that for $m,n>2$ the multidegree polynomial has coefficients $>1$ with respect to the  $\NN^m\oplus \NN^n$-grading. 
		For example, when $m=n=3$ and $r=2$ the multidegree polynomial is 
		\begin{eqnarray*} 
			\mathcal{C}(R/I_2, t_1,t_2,t_3,s_1,s_2,s_3) \;=\; &t_1^2t_2^2+
			t_1^2t_2t_3+
			t_1^2t_2s_1+
			{\bf  2}t_1t_2t_3s_1+
			t_1t_2s_1^2+
			t_1^2s_1s_2+
			{\bf  2}t_1t_2s_1s_2+ \\
			&{\bf  2}t_1s_1s_2s_3+
			t_1s_1^2s_2+
			s_1^2s_2s_3+
			s_1^2s_2^2+
			\cdots\mbox{  symmetric terms}. 
		\end{eqnarray*} 
	\end{remark} 
	%
	%
	%
	%
	%
	%
	%
	%
	%
	%
	%
	%
	%
	%
	%
	%
	%
	%
	%
	%
	%
	%
	%

	\section*{Acknowledgments}
	
	Caminata and Conca were partially supported by PRIN 2020355B8Y ``Squarefree Gr\"obner degenerations, special varieties and related topics" and by GNSAGA-INdAM. 
	Caminata was supported by the GNSAGA-INdAM grant ``New theoretical perspectives via Gr\"obner bases'' and by the European Union within the program NextGenerationEU.
	Cid-Ruiz was partially supported by an FWO Postdoctoral Fellowship (1220122N).
	This work was supported by the “National Group for Algebraic and Geometric Structures, and their Applications” (GNSAGA – INdAM).
	Cid-Ruiz is very appreciative of the hospitality offered by the mathematics department of the University of Genova, where part of this work was done.
	We thank the reviewers for carefully reading our paper and for their comments and corrections.

	

\end{document}